\documentclass[11pt,a4paper,psamsfonts]{amsart}

\usepackage{amsmath}
\usepackage{fancyhdr} 
\usepackage{appendix}
\usepackage{amssymb,amscd,amsxtra,calc}
\usepackage{mathrsfs}
\usepackage{cmmib57}
\usepackage{multirow}
\usepackage[all]{xy}
\usepackage{longtable}
\usepackage{enumitem}
\usepackage{mathtools}
\usepackage{comment}
\usepackage{tikz-cd}

\usepackage[colorlinks,linkcolor=blue,anchorcolor=blue,citecolor=green,backref=page]{hyperref}

\theoremstyle{plain}
    \newtheorem{thm}{Theorem}[section]
    
    \newtheorem{claim}[thm]{Claim}
    
    \newtheorem{corollary}[thm]{Corollary}
    \newtheorem{lemma}[thm]{Lemma}
    \newtheorem{proposition}[thm]{Proposition}
    \newtheorem{question}[thm]{Question}
    \newtheorem{theorem}[thm]{Theorem}

\theoremstyle{definition}
    \newtheorem{example}[thm]{Example}
    \newtheorem{definition}[thm]{Definition}
    
    \newtheorem*{notation*}{Notation and Terminology}
    \newtheorem{remark}[thm]{Remark}
\theoremstyle{remark}

\newcommand{\C}{\mathbb{C}}

\newcommand{\Q}{\mathbb{Q}}

\newcommand{\R}{\mathbb{R}}
 
\newcommand{\Z}{\mathbb{Z}}

\newcommand{\alb}{\operatorname{alb}}

\newcommand{\diag}{\operatorname{diag}}

\newcommand{\id}{\operatorname{id}}

\newcommand{\Ker}{\operatorname{Ker}}

\newcommand{\NS}{\operatorname{NS}}

\newcommand{\Prep}{\operatorname{Prep}}

\newcommand{\rank}{\operatorname{rank}}

\newcommand{\Alb}{\operatorname{Alb}}

\newcommand{\Pic}{\operatorname{Pic}}

\usepackage[colorinlistoftodos]{todonotes}

\newcommand{\nc}{\newcommand}
\nc{\cH}{{\mathcal H}}
\nc{\cA}{{\mathcal A}}
\nc{\cG}{{\mathcal G}}
\nc{\cC}{{\mathcal C}}
\nc{\cO}{{\mathcal O}}
\nc{\cI}{{\mathcal I}}
\nc{\cB}{{\mathcal B}}
\nc{\cY}{{\mathcal Y}}
\nc{\cK}{{\mathcal K}}
\nc{\cX}{{\mathcal X}}
\nc{\cS}{{\mathcal S}}
\nc{\cE}{{\mathcal E}}
\nc{\cF}{{\mathcal F}}
\nc{\cZ}{{\mathcal Z}}
\nc{\cQ}{{\mathcal Q}}
\nc{\cN}{{\mathcal N}}
\nc{\cP}{{\mathcal P}}
\nc{\cL}{{\mathcal L}}
\nc{\cM}{{\mathcal M}}
\nc{\cT}{{\mathcal T}}
\nc{\cW}{{\mathcal W}}
\nc{\cU}{{\mathcal U}}
\nc{\cJ}{{\mathcal J}}
\nc{\cV}{{\mathcal V}}

\nc{\bH}{{\mathbb H}}
\nc{\bA}{{\mathbb A}}
\nc{\bG}{{\mathbb G}}
\nc{\bC}{{\mathbb C}}
\nc{\bO}{{\mathbb O}}
\nc{\bI}{{\mathbb I}}
\nc{\bB}{{\mathbb B}}
\nc{\bY}{{\mathbb Y}}
\nc{\bK}{{\mathbb K}}
\nc{\bX}{{\mathbb X}}
\nc{\bS}{{\mathbb S}}
\nc{\bE}{{\mathbb E}}
\nc{\bF}{{\mathbb F}}
\nc{\bZ}{{\mathbb Z}}
\nc{\bQ}{{\mathbb Q}}
\nc{\bN}{{\mathbb N}}
\nc{\bP}{{\mathbb P}}
\nc{\bL}{{\mathbb L}}
\nc{\bM}{{\mathbb M}}
\nc{\bT}{{\mathbb T}}
\nc{\bW}{{\mathbb W}}
\nc{\bU}{{\mathbb U}}
\nc{\bD}{{\mathbb D}}
\nc{\bJ}{{\mathbb J}}
\nc{\bV}{{\mathbb V}}
\nc{\bbZ}{{\mathbb Z}}
\nc{\bR}{{\mathbb R}}
\nc{\fr}{{\rightarrow}}
\nc{\co}{{\nabla}}

\nc{\cu}{{\overlineline{\nabla}}}

\title [Essential dimensions]{Essential dimensions of polarized endomorphisms of abelian varieties}

\author{Yujie Luo, Keiji Oguiso and De-Qi Zhang}
\date{}

\address{Department of Mathematics, National University of Singapore, Singapore 119076, 
Singapore}

\email{yujieluo96@gmail.com~and~lyj96@nus.edu.sg}

\address{Mathematical Sciences, the University of Tokyo, 
Japan, 
and National Center for Theoretical Sciences, 
Mathematics Division, National Taiwan University, 
Taipei, Taiwan}

\email{oguiso@ms.u-tokyo.ac.jp}

\address{Department of Mathematics, National University of Singapore, Singapore 119076, 
Singapore}

\email{matzdq@nus.edu.sg}


\begin{document}

\begin{abstract}
Let $f$ be a polarized endomorphism of an abelian variety $A$. Koll\'ar and Zhuang asked whether the essential dimension $\mathrm{ed}(f)$ equals $\dim(A)$. We provide counterexamples to this question. Instead, we prove that, under the hypothesis that every subtorus of $A$ is $f$-preperiodic up to translation (a condition arising from the dynamical Manin--Mumford conjecture), we have $\mathrm{ed}(f^s)=\dim(A)$ for some integer $s>0$. Our examples also show the necessity of both the hypothesis and iteration. We also give an affirmative answer to Koll\'ar and Zhuang's original question when $A$ is a simple abelian surface and $f$ is not $2$-polarized.
\end{abstract}

\maketitle

\tableofcontents


%
%
%
%
\section{Introduction}

Throughout this article, we work over an algebraically closed field of characteristic zero.

\medskip

The essential dimensions of algebraic objects, first introduced in a modern context by Buhler and Reichstein \cite{BR97, BR99}, have been studied for decades. For more on the history, we refer the reader to Reichstein's ICM talk \cite{Reic10}; see also the introduction in recent papers \cite{FW19,FKW24}. In this article, we are particularly interested in the essential dimensions of finite morphisms between algebraic varieties. Recall that the \emph{essential dimension of a generically finite (dominant rational) map} $f: X \dasharrow Y$, denoted by $\mathrm{ed}(f)$, is the smallest integer
$d$ such that $f$ is birational to the pull-back of a map of varieties of dimension $d$ \cite{BR97}. When $f$ is clear from the context, we also denote $\mathrm{ed}(f)$ by $\mathrm{ed}(X/Y)$. A generically finite map $f: X\dasharrow Y$ is called \emph{incompressible} if its essential dimension is equal to $\dim(X)$ \cite{BR97,FKW24,KZ}.

Recently, Koll\'ar and Zhuang proved that the multiplication-by-$m$ map $m_A: A \to A$ of an abelian variety $A$ is incompressible for $m\geq 2$ \cite[Theorem~1]{KZ}, which was conjectured by Brosnan \cite[Conjecture~6.1]{FS22}.

In this paper, we identify a Cartier divisor $D$ with its corresponding line bundle $\mathcal{O}(D)$. Recall that an endomorphism $f$ of a projective variety $X$ is \emph{$q$-polarized}, or simply \emph{polarized} (with respect to an ample line bundle $H$), if $f^*H\sim qH$ for some integer $q>1$. Note that multiplication-by-$m$ map $m_A: A \to A$ of an abelian variety $A$ is $m^2$-polarized for $m>1$. 

Koll\'ar and Zhuang raised the following question, which can be viewed as a generalization of \cite[Theorem~1]{KZ}:

\begin{question}[{\cite[Question~19]{KZ}}]\label{ques: polarized incompressible}
    Is every polarized endomorphism incompressible?
\end{question}

As mentioned in \cite{KZ}, Fakhruddin pointed out that Question~\ref{ques: polarized incompressible} has a negative answer in positive characteristic, even after iteration (see Example~\ref{ex: positive characteristic iteration}). 

Note that if $X$ is a normal projective variety that admits a polarized endomorphism, then the Kodaira dimension $\kappa(X)\leq 0$ \cite[Theorem~1.3]{NZ}. Before further discussing Question~\ref{ques: polarized incompressible}, we make the following remarks, which suggest that polarized endomorphisms are built from those on abelian varieties and rationally connected varieties. In this paper, we focus on abelian varieties.

\begin{remark}\label{rem: kernel and incompressible}
We consider a smooth projective variety $X$ admitting a $q$-polarized endomorphism $f$. Recall that the \emph{rank} of an abelian group $G$ is the minimum number of generators of $G$ when $G\neq \{0\}$, and we set $\rank \{0\}=0$.
\begin{enumerate}
    \item Suppose that $f : X \to X$ is a finite abelian cover with Galois group $G$. Then $\mathrm{ed}(f) \leq \text{rank}(G)$ (\cite[Theorem~6.1]{BR97}, see also Lemma~\ref{lem: facts of essential dimension}(4)). Thus, a necessary condition for a self-isogeny $f$ of an abelian variety $A$ to be incompressible is $\rank \Ker(f) \ge \dim(A)$. However, this is not a sufficient condition, see Example~\ref{ex: counterexample for stronger version of KZ question} and Remark~\ref{rem: rank is not ed}.
    \item 
    Suppose that $X$ has maximal Albanese dimension. Then the Albanese map $\alb_X: X \to \Alb(X)$ is an isomorphism, hence $X$ is an abelian variety (see Lemma~\ref{lem: maximal Albanese dimension}).
    \item If $X$ is non-uniruled, by Proposition \ref{thm: quasi-abelian and non-uniruled}, one can lift $f$ to a polarized endomorphism $f_A$ on an abelian variety $A$ via a quasi-\'etale cover $\sigma: A \to X$, such that $\mathrm{ed}(f^s)\geq \mathrm{ed}(f_A^s)$ for all $s \ge 1$.
    If $X$ is uniruled, by \cite[Proposition 1.6]{MZ18}, one can descend $f$ to a polarized endomorphism $f_Y$ on $Y$ via a special maximal rationally connected fibration $X\dasharrow Y$ with $Y$ being non-uniruled so that Proposition \ref{thm: quasi-abelian and non-uniruled} is applicable again.
    \item Essential dimensions of endomorphisms of an abelian variety $A$ are translation invariant (and in fact, birational invariant): $\mathrm{ed}(f)=\mathrm{ed}(t_x\circ f)=\mathrm{ed}(f\circ t_y)$ for any $x,y\in A$ (see Lemma~\ref{lem: facts of essential dimension}(3)). Thus it suffices to consider essential dimensions of self-isogenies of $A$. Indeed, both $t_x\circ f$ and $f\circ t_y$ are still $q$-polarized.
\end{enumerate}
\end{remark}

There are many counterexamples to Question~\ref{ques: polarized incompressible} as we shall see later. The following example, which came from a communication with S. Meng and J. Xie, shows that polarized endomorphisms need not be incompressible.

\begin{example}\label{ex: counterexample for KZ question}
Let $E_n:=\mathbb{C}/(\mathbb{Z}\oplus ni\mathbb{Z})$ be an elliptic curve and $h_n: z\to niz$ an isogeny of $E_n$. Consider the self-isogeny $f(z_1,z_2)=(z_2,h_n(z_1))$ of $E_n\times E_n$. Since $f^2=h_n\times h_n$ is $n^2$-polarized, $f$ is $n$-polarized (cf. \cite[Proof of Theorem~2.7, Note 1]{Zh10}). Note that $\Ker(f) \cong \Ker(h_n)= \langle i/n\rangle \cong \mathbb{Z}/n^2\mathbb{Z}$. Hence $\mathrm{ed}(f)\leq\rank \Ker(f)=1$. Moreover, $f^2$ factors through a multiplication-by-$n$ map, hence $\mathrm{ed}(f^s)=2$ for $s\geq 2$ (cf. Lemma~\ref{lem: facts of essential dimension}(3)).
\end{example}

We remark that using the same idea, one can provide counterexamples to Question~\ref{ques: polarized incompressible} when the base varieties are rationally connected  (see Example~\ref{ex: counterexample rationally connected}).

Recall that for an endomorphism $f$ on a variety $X$, we say a subvariety $Y\subset X$ is \emph{$f$-preperiodic} if $f^s(Y)=f^r(Y)$ for some non-negative integers $r>s$. Moreover, if $X$ is an abelian variety, we say $Y$ is \emph{$f$-preperiodic up to translation} (which, by Lemma~\ref{lem: numerically parallel abelian varieties}, is equivalent to being \emph{numerically $f$-preperiodic}) if $f^s(Y)$ is a translation of $f^r(Y)$ for some non-negative integer $r>s$. We say $Y$ is \emph{$f$-periodic} (resp. \emph{$f$-periodic up to translation}) if we can take $s=0$.

In view of Example \ref{ex: counterexample for KZ question}, it is natural to ask whether a polarized endomorphism becomes incompressible after some iteration. However, counterexamples still exist. We construct a polarized self-isogeny $f$ on a non-simple abelian surface $A$ for which no iteration is incompressible by finding a one-dimensional subtorus $E$ that is not $f$-preperiodic and such that the degree of the restricted map $f^s|_E$ is one for all positive integers $s$; we refer the reader to Example~\ref{ex: counterexample for stronger version of KZ question} and Proposition \ref{prop: essential dimension be one criterion} for details.

In view of Remark~\ref{rem: kernel and incompressible} and Examples~\ref{ex: counterexample for KZ question} and \ref{ex: counterexample for stronger version of KZ question}, it is natural to ask the following question: {\it For a polarized self-isogeny $f$ of an abelian variety $A$, if every subtorus of $A$ is $f$-preperiodic (this holds when $f = m_A$ or $A$ is simple), is $f^s$ incompressible for some positive integer $s$?}

In this paper, we give an affirmative answer to the above-mentioned question.

\begin{theorem}[{Theorem~\ref{thm: main with all subtori f numerically periodic}}]\label{thm: ed main}
    Let $A$ be an abelian variety, and let $f$ be a polarized endomorphism of $A$. Suppose that every subtorus of $A$ is $f$-preperiodic up to translation. Then $f^s$ is incompressible for some positive integer $s$. 
\end{theorem}

We remark that the assumption `every subtorus of $A$ is $f$-preperiodic up to translation' naturally arises from the original version of the algebraic dynamical Manin--Mumford conjecture \cite[Conjecture~2.1]{GTZ}. We say $f$ satisfies the \emph{dynamical Manin--Mumford condition} if the following holds: for any subvariety $Y$ of $X$, the subvariety $Y$ is $f$-preperiodic if and only if the $f$-preperiodic points of $X$ are Zariski dense in $Y$ (see Definition~\ref{defn: dynamical Manin Mumford condition}). Our assumption in Theorem~\ref{thm: ed main} is equivalent to the dynamical Manin--Mumford condition up to translation (see Lemma \ref{lem: dynamical Manin Mumford condition on abelian varieties}). Hence, the following holds (Theorem~\ref{thm: dynamical Manin Mumford condition}): 

\medskip

\emph{If a polarized self-isogeny $f$ of an abelian variety satisfies the dynamic Manin--Mumford condition, then $f^s$ is incompressible for some positive integer $s$}.

\medskip

We also remark that if we only require $f$ to be int-amplified in Theorem~\ref{thm: ed main}, then counterexamples exist (see Example~\ref{ex: cyclic after iteration example int-amplified}).

Under the assumption that $A$ is a simple abelian variety, Koll\'ar and Zhuang \cite[Corollary 17]{KZ} show that $\mathrm{ed}(f) = \min\{\dim(A), \rank \Ker(f)\}$ when $\deg(f)$ is coprime to $(\dim(A))!$. As a complementary result, Theorem~\ref{thm: ed main} implies the following since the $f: A \to A$ below maps every subtorus of $A$ to a translation of itself; see Lemma \ref{lem: numerically parallel abelian varieties}.

\begin{corollary}\label{cor: simple main}
    Let $A$ be an abelian variety that is isogenous to the product of pairwise non-isogenous simple abelian varieties. Let $f$ be a polarized endomorphism of $A$. Then $f^s$ is incompressible for some positive integer $s$.
\end{corollary}

One of the key ingredients in the proof of Theorem~\ref{thm: ed main} is the Jordan property for birational automorphism groups of rationally connected varieties of bounded dimension, which was proved by Prokhorov and Shramov (\cite[Theorems~1.8 and 1.10]{PS14}) assuming the Borisov–Alexeev–Borisov conjecture (now a theorem of Birkar \cite{Bir21}).

When $f$ is $q$-polarized, our proof of Theorem~\ref{thm: ed main} yields an explicit bound for the iteration index (see Theorem~\ref{thm: main with all subtori f numerically periodic}): $$s>\log_q(J)+\dim(A)\cdot \log_q |K(\mathcal{L})|.$$ Here $J$ is a constant depending only on the dimension of $A$, $f$ is $q$-polarized with respect to an ample line bundle $\mathcal{L}$, and $|K(\mathcal{L})|=\chi(\mathcal{L})^2=(\frac{1}{\dim(A)!}\cdot \mathrm{Vol}(\mathcal{L}))^2 = (\frac{1}{\dim(A)!}\cdot \mathcal{L}^{\dim(A)})^2$ is the degree of the polarization map induced by $\mathcal{L}$.

\medskip

In particular, for abelian varieties that are isogenous to products of some elliptic curves, we have the following.

\begin{theorem}[{Corollary~\ref{cor: essential dimension isogenous to product elliptic curves periodic}}]\label{thm: essential dimension isogenous to product elliptic curves intro}
    There exists a positive integer $s$ depending only on the dimension $n$ that satisfies the following.

    Let $A$ be an abelian variety that is isogenous to the product $E_1\times \cdots \times E_n$ of elliptic curves, and let $f$ be a polarized endomorphism of $A$. Assume that every one-dimensional subtorus of $A$ is $f$-stable up to translation. Then $f^s$ is incompressible. Furthermore, for any prime $p$ dividing $\deg(f)$, the endomorphism $f$ itself satisfies
    $$\mathrm{ed}(f)\geq \frac{p-1}{p} \dim(A).$$
\end{theorem}

Let $\NS(A)$ be the N\'eron-Severi group of $A$ and ${\NS(A)_{K}} = \NS(A) \otimes K$ for $K = \Q, \R$. We remark that the assumption in Theorem~\ref{thm: essential dimension isogenous to product elliptic curves intro} is equivalent to that $f^*|_{\NS(A)_{\Q}}$ is a scalar multiplication; see Lemma \ref{lem: basic property of polarized end of product abelian}. Note also that the assumption in Theorem~\ref{thm: essential dimension isogenous to product elliptic curves intro} is necessary; see Example~\ref{ex: counterexample for stronger version of KZ question} and Proposition \ref{prop: essential dimension be one criterion} for counterexamples without this condition.

We also remark that the constant $s$ in Theorem~\ref{thm: essential dimension isogenous to product elliptic curves intro} can be chosen as $\lceil\log_2 J\rceil$, where $J$ is the Jordan constant of the birational automorphism groups of rationally connected varieties of dimension at most $n-1$ (see \cite[Theorems~1.8 and 1.10]{PS14}).

For the pair $(A,f)$ in Theorem~\ref{thm: essential dimension isogenous to product elliptic curves intro}, if $\deg(f)$ has a prime divisor larger than $\dim(A)$, then $f$ is incompressible. As another simple corollary, we have $\mathrm{ed}(f)\geq \frac{1}{2}\dim(A)$ for all such pairs $(A,f)$.

\medskip

In dimension two, we prove that polarized endomorphisms of simple abelian surfaces are incompressible, with the possible exception of one case that is elaborated in Remark \ref{rem:2-pol}.

\begin{theorem}\label{thm: simple abelian surface uniform index}
    Let $A$ be a simple abelian surface and let $f$ be a polarized endomorphism of $A$. Then $f$ is either incompressible or $2$-polarized.
\end{theorem}

\medskip

Recall that for a finite abelian group $G$ and a prime $p$, the \emph{$p$-rank} of $G$ is defined as $\rank_p G:=\rank (G/pG)$. If $G_p$ denotes the Sylow $p$-subgroup of $G$, then $\rank_p G=\rank (G_p)$.

Motivated by Remark~\ref{rem: kernel and incompressible} and Example \ref{ex: counterexample for KZ question}, it is also natural to ask whether the rank (or $p$-rank) of the kernel of a polarized self-isogeny $f$ of an abelian variety $A$ is at least $\dim(A)$ after iteration. We answer this affirmatively for all abelian varieties of arbitrary dimension.

\begin{theorem}[Theorem~\ref{thm: rank main 1}]\label{thm: rank main}
Let $f$ be a polarized self-isogeny of an abelian variety $A$. Then for any prime $p\mid \deg(f)$, we have $\rank_p \Ker(f^s)\geq \dim(A)$ for some positive integer $s$.
\end{theorem}

As a direct consequence of Theorem~\ref{thm: rank main}, $\rank \Ker(f^s)\geq \dim(A)$ for some positive integer $s$. We note that the above rank bound is optimal (see Example~\ref{ex: optimal rank bound for rank main}). Moreover, if $A$ is principally polarized by a numerically $f$-periodic line bundle $\mathcal{L}$, then $\rank \Ker (f)\geq \dim(A)$ (see Corollary~\ref{cor: explicit index s principal}).

In dimension two, we also obtain the following uniform bound for the index $s$ in Theorem~\ref{thm: rank main}.

\begin{theorem}\label{thm: rank main uniform}
Let $A$ be an abelian surface and let $f$ be any polarized self-isogeny of $A$. Then the following statements hold. 
\begin{enumerate}
    \item There exists a positive integer $s$ depending only on $A$ such that $\rank \Ker (f^s)\geq 2$.
    \item If $A$ is simple, then for any prime $p\mid \deg(f)$, we have $\rank_p\Ker(f)\geq 2.$
\end{enumerate}
\end{theorem}

We remark that the proof of Theorem~\ref{thm: rank main uniform} differs substantially from that of Theorem~\ref{thm: rank main}, and it requires a detailed analysis of both the lattice of $A$ and the explicit structure of polarized endomorphisms of $A$ (see \cite{Alb34, Alb35, Shi63, SM74, Mur84, LB92}). Moreover, Theorem~\ref{thm: rank main uniform} actually holds over an algebraically closed field of characteristic zero by the Lefschetz principle, although our proof is over $\C$. We also remark that Theorem~\ref{thm: rank main uniform} plays a key role in the proof of Theorem~\ref{thm: simple abelian surface uniform index}.

\medskip

We end the introduction with a result for non-uniruled varieties.
A normal projective variety $V$ is {\it $Q$-abelian} if there is a quasi-\'etale finite surjective morphism $A \to V$ from an abelian variety $A$ (cf. \cite[Definition~2.13]{NZ}).

\begin{proposition}
\label{thm: quasi-abelian and non-uniruled}
Let $X$ be a non-uniruled normal projective variety and $f$ a polarized endomorphism of $X$. Then
$X$ is $Q$-abelian and $f$ lifts to a polarized endomorphism $f_A: A = A_1 \to A = A_2$ on an abelian variety $A$ via a quasi-\'etale finite Galois cover $\sigma: A \to X$, so that $A_1$ is the normalization of the fibre product of $f$ and $\sigma: A_2 \to X$. 

\par
In particular, $\mathrm{ed}(f^s) \ge \mathrm{ed}(f_A^s)$ for all $s \ge 1$.
\end{proposition}

\vspace{2mm}

{\bf Acknowledgements.}
The authors thank S. Meng and Z. Zhuang for many helpful discussions and comments, and J. Xie, S.-W. Zhang, and G. Zhong for valuable suggestions and interest in this work. Luo, Oguiso, and Zhang are supported by a PTA fellowship of NUS; JSPS Grant-in-Aid (A) 25H00587, 25K21992 and NCTS scholar program; and ARF of NUS: A-8002487-00-00, respectively.
\medskip

\section{Preliminary results}
We adopt the standard notations as in \cite{Ful}, \cite{Har}, \cite{Laz} and \cite{Mum74}. We start with a definition of essential dimension in the classical context.

\begin{definition}[cf. {\cite[Definition~2.1]{BR97}}]
    Let $E/F$ be a finite field extension over a base field $k$. We say that $E/F$ is \emph{defined over} a subfield $F_0$ of $F$ if there exists an extension $E_0/F_0$ such that $[E_0:F_0]=[E:F]$, $E_0 \subseteq E$ and $E_0F = E$. The \emph{essential dimension} of $E/F$, which we will usually abbreviate as $\mathrm{ed}_k(E/F)$ (or $\mathrm{ed}(E/F)$ when $k$ is clear from the context), is the minimal value of $\mathrm{trdeg}_k(F_0)$ as $F_0$ ranges over all fields for which $E/F$ is defined over $F_0$. The \emph{essential dimension of a generically finite dominant rational map} $f: X\dasharrow Y$, denoted by $\mathrm{ed}(f)$, is defined as $\mathrm{ed}_k(K(X)/K(Y))$. Here $K(X)$ is the function field of an algebraic variety $X$.
\end{definition}

We collect the following facts on essential dimensions defined above.

\begin{lemma}\label{lem: facts of essential dimension}
We fix the base field $k$ of characteristic zero.
\begin{enumerate}
    \item If two generically finite dominant rational maps are birationally equivalent, then they have the same essential dimension. In particular, the essential dimension of a birational map is zero.
    \item Let $K/L$ be a finite field extension over $k$, and let $N/L$ be the Galois closure of $K$ over $L$. Then $\mathrm{ed}(N/L)=\mathrm{ed}(K/L)$.
    \item Let $f: X\dasharrow Y$ and $g: Y\dasharrow Z$ be generically finite dominant rational maps. Then $\max\{\mathrm{ed}(g),\mathrm{ed}(f)\}\leq \mathrm{ed} (g\circ f)$. In particular, $\mathrm{ed}(g_2\circ f\circ g_1)=\mathrm{ed}(f)$, where $g_1$ (resp. $g_2$) is a birational map on $X$ (resp. $Y$).
    \item Let $f: X \dasharrow Y$ be a generically finite dominant rational map such that the induced field extension $K(X)/K(Y)$ is abelian with Galois group $G$. Then $\mathrm{ed}(f) \leq \rank(G)$.
    \item Let $f_i \colon X_i \dasharrow Y$ ($i = 1, 2$) be generically finite dominant rational maps and let $f \colon X = X_1 \times_Y X_2 \dasharrow Y$ be the map induced by the fibre product. Then $\mathrm{ed}(f) \leq \mathrm{ed}(f_1) + \mathrm{ed}(f_2)$. In particular, for generically finite dominant rational maps $f_i:X_i\dashrightarrow Y_i$ ($i=1,2$), we have $\mathrm{ed}(f_1\times_k f_2:X_1\times_k X_2\dashrightarrow Y_1\times_k Y_2)\leq \mathrm{ed}(f_1)+\mathrm{ed}(f_2)$.
\end{enumerate}
\end{lemma}

\begin{proof}
    (1) is clear from the definition of essential dimensions. (2) is from \cite[Lemma~2.3]{BR97}. (4) follows from \cite[Lemma~6.1]{BR97}, and (5) is from {\cite[Theorem~15]{KZ}} and that $f_1\times_k f_2$ is birational to the fibre product of $f_1 \times \id$ and $\id \times f_2$ (for the second part). For (3), it is clear that $\mathrm{ed}(f)\leq \mathrm{ed}(g\circ f)$. It suffices to prove the claim that $\mathrm{ed}(g)\leq \mathrm{ed}(g\circ f)$. 
    
    We consider the Galois closure $K(X)^\#$ of $K(X)$ over $K(Z)$. By (2), $\mathrm{ed}(K(X)^\#/K(Z))=\mathrm{ed}(K(X)/K(Z))$. Possibly replacing $K(X)$ with $K(X)^\#$, we may assume that $K(X)/K(Z)$ is Galois with Galois group $G$. Let $H$ be the Galois group of $K(X)/K(Y)$. By \cite[Lemma~2.2]{BR97}, there exists a $G$-invariant subfield $N$ of $K(X)$ on which $G$ acts faithfully, such that $K(X) = NK(Z)$ and $\mathrm{trdeg}_k(N)=\mathrm{ed}_k(K(X)/K(Z))$. Denote $N^G$ by $F_0$ and $E_0 = N^H$. Then $E_0 \subset K(Y)$, $[E_0 : F_0] = [G : H] = [K(Y) : K(Z)]$, and $E_0K(Z)=N^HK(Z)=(NK(Z))^H=K(X)^H=K(Y)$. It follows that $K(X)/K(Z)$ is defined over $F_0$. Thus $$\mathrm{ed}(K(Y)/K(Z))\leq \mathrm{trdeg}_k(F_0)\leq \mathrm{trdeg}_k(N)=\mathrm{ed}_k(K(X)/K(Z)).$$ This proves the claim and also the lemma.
\end{proof}


\begin{definition}\label{defn: q-polarized}
    Let $f: X\to X$ be a finite surjective endomorphism of a normal projective variety $X$.

    We say $f$ is \emph{$q$-polarized} or simply \emph{polarized}, if there exists an ample Cartier divisor (or equivalently, ample $\mathbb{R}$-Cartier divisor, see \cite[Lemma~3.5]{MZ18}) $H$ on $X$ such that $f^* H\sim qH$ (or equivalently, $f^*H\equiv qH$, see \cite[Lemma~2.3]{NZ}) for some integer (or equivalently, rational number, see \cite[Lemma~3.5]{MZ18}) $q > 1$.

    We say $f$ is \emph{numerically $q$-polarized} or simply \emph{numerically polarized}, if there exists an ample $\R$-Cartier divisor $H$ on $X$ such that $f^*H\equiv qH$ for some real number $q > 1$. Moreover, if $f$ is numerically $q$-polarized, then 
    $q^{\dim(X)} = \deg f$ (an integer) by the projection formula. So 
    $f^{\dim(X)}$ is $q^{\dim(X)}$-polarized.

    The above notions behave well under iterations. To be more specific, if $f^s$ is numerically polarized (resp. $q^s$-polarized for some integer $q > 1$) for some integer $s \ge 1$, then $f^t$ is numerically polarized for any integer $t \ge 1$ (resp. $f$ is $q$-polarized); see \cite[Proof of Theorem~2.7, Note 1]{Zh10}.
\end{definition}

\begin{definition}\label{defn: numerically parallel and periodic}
Let $X$ be a normal projective variety.

An $r$-cycle $D$ on $X$ is \emph{weakly numerically equal} to zero, denoted as $D \equiv_w 0$ if $D . L_1 \dots L_r = 0$ for all Cartier divisors $L_j$.
For two $r$-cycles $D_j$ ($j=1,2$), by $D_1 \equiv_w D_2$ we mean $D_1 - D_2 \equiv_w 0$. By \cite[Lemma 3.2]{Zh16}, a Cartier divisor $D \equiv_w 0$ if and only if $D \equiv 0$ (numerical equivalence).

We say two $r$-cycles $D_j$ are \emph{numerically parallel} if $D_1 \equiv_w cD_2$ for some $c \ne 0$. Let $f: X \to X$ be a finite surjective endomorphism. A subvariety $D$ is \emph{numerically $f$-periodic} (resp. {\it $f$-preperiodic}) if $f^s(D)$ and $D$ (resp. $f^r(D)$ and $f^s(D)$) are numerically parallel (viewed as cycles) for some positive integer $s$ (resp. non-negative integers $r\neq s$). In particular, if $f(D)$ is numerically parallel to $D$, then we say $D$ is \emph{numerically $f$-stable}. A Cartier divisor $L$ is \emph{numerically $f$-periodic} if $(f^s)^*(L)$ (or equivalently, $(f^s)_*L$) is numerically parallel to $L$ for some positive integer $s$; see the $(*)$ in Lemma \ref{lem: numerically f-periodic elliptic curve and quotient} for the equivalence.
\end{definition}

\begin{lemma}\label{lem: eigenvalues of polarized map}
    Let $X$ be an abelian variety, and let $f$ be a surjective endomorphism of $X$. Then the following statements are equivalent:
    \begin{enumerate}
        \item $f$ is numerically $q$-polarized.
        \item The induced action $f^*$ on $H^{1,0}(X,\mathbb{C})$ is diagonalizable whose eigenvalues are of the same modulus $\sqrt{q}>1$.
    \end{enumerate}
\end{lemma}

\begin{proof}

Since $X$ is abelian, we have $H^{1,1}(X)=H^{1,0}(X)\otimes H^{0,1}(X)$. 

Suppose that $f$ is numerically $q$-polarized. By \cite[Lemma~2.3(1) as well as the corresponding proof]{NZ}, there exists a positive number $q$ such that the action $\frac{1}{\sqrt{q}}f^*$ on $H^{0,1}(X)$ is unitary. This proves $(2)$.

Now, we prove the opposite direction. Suppose that $f^*$ is diagonalizable on $H^{1,0}(X)$ with eigenvalues $\lambda_i$ of the same modulus $\sqrt{q}$ for $1\leq i\leq \dim(X)$. Then the induced action of $f^*$ on $H^{1,1}(X)$ is diagonalizable with the eigenvalues $\lambda_i \bar \lambda_j$ for $1\leq i,j\leq \dim(X)$ that are of the same modulus $q\in \mathbb{R}$. By \cite[Proposition~2.9]{MZ18}, we have that $f$ is numerically $q$-polarized.
\end{proof}

Recall that if $B$ is a subvariety of an abelian variety $A$, then the Kodaira dimension $\kappa(B)$ of (any smooth model of) $B$ satisfies $\kappa(B) \ge 0$, and $\kappa(B) = 0$ if and only if $B$ is a translation of a subtorus of $A$; see \cite[Theorem~10.3]{Ue75} or \cite[Corollary~3.5]{Mor87}.

Lemma~\ref{lem: numerically parallel abelian varieties} below says that if two subtori of an abelian variety are numerically parallel, then they are identical.

\begin{lemma}\label{lem: numerically parallel abelian varieties}
    Let $A$ be an abelian variety. Let $B, B'$ be two subvarieties of $A$ of the same dimension such that the Kodaira dimensions $\kappa(B)=\kappa(B')=0$. Then:
\begin{enumerate} 
\item[(1)]
Both $B$ and $B'$ are translations of subtori of $A$. 
\item[(2)] 
Let $f: A \to A$ be a finite surjective endomorphism. Then $f(B)$ is a translation of a subtorus.

\par \noindent
Moreover, the following are equivalent.

\item[(3i)] $B'\equiv_w cB$ for some $c \ne 0$.
\item[(3ii)] 
$B$ and $B'$ are translations of the same subtorus in $A$ (In particular, $B = B'$ when $0_A\in B\cap B'$).
\item[(3iii)] 
$B'$ is a fiber of the quotient map $A\to Y$ whose fibers are translations of $B$.
\end{enumerate}
\end{lemma}

\begin{proof}
By \cite[Corollary~3.5]{Mor87}, both $B$ and $B'$ are translations of some subtori in $A$. (1) follows.

By the ramification divisor formula, we have $0 = \kappa(B) \ge \kappa(f(B)) \ge 0$, so $\kappa(f(B)) = 0$ and (2) follows; see \cite[Corollary~3.5]{Mor87}.

Now we prove the equivalence part of the lemma.
Clearly (3ii) and (3iii) are equivalent, and they easily imply (3i). Now we assume (3i). We will prove (3ii). Possibly replacing $B$ with its translation, we may assume that $B$ is a subgroup, hence a subtorus of $A$. Consider the quotient morphism $\pi: A\to A/B$. 
    If $\pi(B')$ is a point, then $B'$ is a fiber and we are done.
    Suppose to the contrary that $C:=\pi(B')$ is positive-dimensional.
    Take a very ample divisor $H_Y$ on $Y$ away from the point $\pi(B)$, and ample divisors $H_j$ on $A$.
    Then for $r = \dim(B)$ and ample divisors $H_1,\cdots, H_{r-1}$, we have $$0 = cB . (\pi^*H_Y). H_1 \cdots H_{r-1} = B' . (\pi^*H_Y) . H_1 \cdots H_{r-1},$$ where the right hand side is positive because $B' . \pi^*H_Y$
    is a positive cycle dominating nonzero cycle $C \cap H_Y$ via $\pi$.
    This is a contradiction.
\end{proof}

\begin{lemma}\label{lem: uniform stable index}
    Let $A$ be an abelian variety, and let $f$ be a numerically $q$-polarized self-isogeny of $A$. Then the following statements hold.
    \begin{enumerate}
        \item Let $B$ be a subtorus of $A$. If $B$ is numerically $f$-preperiodic, then $B$ is $f$-periodic.
        \item Suppose that every subtorus of $A$ is numerically $f$-preperiodic. Then there exists a positive integer $v$ such that every subtorus of $A$ is $f^v$-stable.
    \end{enumerate}
\end{lemma}

\begin{proof}
    Let $N^k(A)$ be the $\R$-linear space of codimension $k$ algebraic cycles of $A$ modulo numerical equivalence, and let $TN^k(A)\subset N^k(A)$ be the $\R$-linear subspace generated by the classes of codimension $k$ subtori of $A$. By \cite[Theorem~3.5]{Klei68}, $N^k(A)$ is a finite dimensional $\R$-linear space. 
    Let $N_w^k(A)$ be the group of $k$-cycles modulo weak numerical equivalence, and $N^k(A) \to N_w^k(A)$ and $TN^k(A) \to TN_w^k(A)$ the natural surjective homomorphisms. By the projection formula, we have $f_*f^*
    = \deg(f) \cdot \id$ on $TN^k(A)$, where $\deg(f) = q^{\dim(A)}$. See \cite[\S 2.3]{Zh10} for the pullback and intersection of cycles.
    
    For (1), since $B$ is numerically $f$-preperiodic, there are positive integers $r > s$ such that $f^{r}(B)$ and $f^{s}(B)$ are numerically parallel. So $(f^{r})^*B$ and $(f^{s})^*B$ are numerically parallel too. Here we view $B$ (instead of $[B]$) as an element in $TN_w^k(A)$ with $k$ the codimension of $B$ in $A$. Since $f$ is numerically $q$-polarized and $(f^t)^*$ acts as an isomorphism on $TN_w^k(A)$ for any $t$, we have $(f^{u})^*[B]=q^{uk}[B]$ with $u = r-s$. This implies that $B$ is numerically $f$-periodic, hence $f$-periodic by Lemma~\ref{lem: numerically parallel abelian varieties}.

    For (2), we take subtori $T_i$ ($1\leq i\leq m$) such that their classes $[T_i]$ form a finite $\R$-basis in $TN_w^k(A)$. By (1), every subtorus $T_i$ of $A$ is $f$-periodic. This implies that $(f^v)^*|_{TN_w^k(A)}=q^{vk} \cdot \mathrm{id}$ for some $v > 0$. In particular, every subtorus of $A$ is numerically $f^v$-stable, and hence $f^v$-stable by Lemma~\ref{lem: numerically parallel abelian varieties}.
\end{proof}

\begin{lemma}\label{lem: numerically f-periodic elliptic curve and quotient}
Let $f: A \to A$ be a finite surjective endomorphism of an abelian surface, and let $E \subset A$ be an elliptic curve. Then:
\begin{enumerate}
\item 
$f(E)$ is a smooth curve of genus one. 
\item 
Let $\tau_1: A \to Y_1$ and $\tau_2: A \to Y_2$ be the quotient maps with the fibers being translates of $E$ and $f(E)$ respectively.
Then $f$ induces a
surjective morphism $g: Y_1 \to Y_2$ so that $f^*$ takes a fiber of $\tau_2$ to a disjoint union of $\deg g$ of fibers of $\tau_1$. Furthermore,
$$f^*(f(E)) \equiv (\deg g) E, \,\,\,  \deg f|_E = (\deg f)/(\deg g).$$
\item 
$f(E)$ and $E$ are numerically parallel (i.e., $f(E) \equiv aE$ for some $a > 0$) if and only if $E$ and $f^*E$ are numerically parallel, if and only if $f$ permutes fibers of $\tau_1$. In this case, $\tau_1 = \tau_2$.
\item 
$f^i(E)$ is numerically parallel to $f^j(E)$ for some $j > i$ if and only if $E$ and $f^{j-i}(E)$ are numerically parallel.
\end{enumerate}
\end{lemma}

\begin{proof}
(1) follows from Lemma~\ref{lem: numerically parallel abelian varieties}.

For (2), by possibly replacing $E$ with its translation, we may assume that $E$ (hence $f(E)$) is a subgroup of $A$. Now we consider the morphism $g'=\tau_2\circ f: A\to A \to Y_2=A/f(E)$. Since $\Ker  (g') \supset E$, we have an induced surjective isogeny $g: Y_1=A/E \to Y_2=A/f(E)$ and the following commutative diagram
$$\xymatrix{
    A \ar[r]^f \ar[d]_{\tau_1} & A \ar[d]^{\tau_2}\\
    Y_1 \ar[r]^g & Y_2
}$$
In particular, $f^*$ takes a fiber of $\tau_2$ to a disjoint union of $\deg g$ of fibers of $\tau_1$. Now the last part of (2) follows from the projection formula.

For (3)-(4), note that 
$$(*) \,\hskip 1pc 
f_* f^* = (\deg f) \id = f^* f_*$$ on $\Pic (A) \otimes_{\Z} \Q$ (actually true on any projective variety); see \cite[Proposition~3.7(3)]{MZ22}. $(*)$ implies that $f^*f(E)$ is numerically parallel to $E$. Hence $f(E)$ and $E$ are numerically parallel if and only if $E$ and $f^*E$ are numerically parallel. The remaining assertions are clear.
\end{proof}

We denote the exponent of a finite abelian group $G$ as ${\rm exp}(G)$.

\begin{lemma}\label{lem: complex structure of endomorphism ring of abelian variety}
    Let $A$ be an abelian variety of dimension $n$. Let $\lambda$ be a complex number such that the multiplication-by-$\lambda$ map defines a $q$-polarized self-isogeny of $A$. Then $\lambda+\bar\lambda$ is an integer. In particular, $\mathrm{exp}(\Ker(\lambda\cdot))$ divides $q$.
\end{lemma}

\begin{proof}
Write $A=\mathbb{C}^n/\Lambda$, where $\Lambda$ is a lattice generated by $2n$ complex (column) vectors $v_1,\cdots,v_{2n}$. Since the multiplication-by-$\lambda$ map defines a self-isogeny on $A$, we have $\lambda\cdot \Lambda\subset \Lambda$, and $$\lambda\cdot (v_1,\cdots,v_{2n})= (v_1,\cdots,v_{2n})\cdot L$$ for some $2n\times 2n$ integer matrix $L$. Note that the matrix of $\lambda\cdot$ with respect to the standard $\mathbb{R}$-basis of $\mathbb{C}^n$ given by $$\{e_k,ie_k\mid e_k \text{ is the standard Euclidean basis vector whose $k$-th coordinate is 1}\}_{1\leq k\leq n}$$ is $\mathrm{diag}(\rho(\lambda),\cdots,\rho(\lambda))$, where $\rho(a+bi)=\begin{pmatrix}
        a &b\\-b &a
    \end{pmatrix}$. This implies that $$\mathrm{tr}(L)=\mathrm{tr}(\mathrm{diag}(\rho(\lambda),\cdots,\rho(\lambda)))=n(\lambda+\bar\lambda)\in \mathbb{Z}.$$ Let $p_{\lambda}(x)$ be the characteristic polynomial for the matrix $\rho(\lambda)$. Then $$p_{\lambda}(x)=(x-\lambda)(x-\bar \lambda)=x^2-(\lambda+\bar \lambda)+q\in \Q[x].$$ 
    There is a rational number $m$ such that $mp_{\lambda}(x)$ is a primitive element in $\Z[x]$. By Gauss's lemma, $(mp_{\lambda}(x))^n$ is primitive in $\Z[x]$.
    Since $L$ is similar to $\mathrm{diag}(\rho(\lambda),\cdots,\rho(\lambda))$, $$p_{\lambda}(x)^n=\det(xI_{2n}-L)$$ is a monic in $\Z[x]$. This implies that $m=1$. Hence $p_{\lambda}(x)\in \Z[x]$, and $\lambda+\bar\lambda=n_{\lambda}\in \Z$. In particular, $\bar{\lambda}=n_{\lambda}-\lambda$ defines a self-isogeny of $A$. Now, for any $x\in \Ker(\lambda\cdot)$, $0 = \bar\lambda\cdot \lambda\cdot x=|\lambda|^2\cdot x=q\cdot x$. Hence, the exponent $\mathrm{exp}(\mathrm{Ker}(\lambda\cdot))$ divides $q$.
\end{proof}

\begin{lemma}\label{lem: basic property of abelian Neron Severi}
Let $A$ be an abelian variety and let $B$ be any subtorus of $A$. Then we have:
\begin{enumerate}
\item[(1)]
$\NS(B)_{\Q} = (\NS(A)_{\Q})|_B$.
\item[(2)]
Suppose that $A$ is isogenous to $E_1\times \cdots \times E_n$, where $E_i$ is an elliptic curve for each $i$. Then for every subtorus $B$ of codimension $m > 0$, we have a sequence $B = A_m \subset A_{m-1} \subset \cdots \subset A_{1}$ of subtori $A_k$ of codimension $k$.
\end{enumerate}
\end{lemma}

\begin{proof}
For (1), by the complete reducibility of Poincar\'e, there is a subtorus $B'$ such that the addition map $\mu: B \times B' \to A$ is an isogeny (of the same Picard number). Let $\pi_B: B \times B' \to B$ be the projection.
For any $\Q$-Cartier divisor $D$ on $B$, we have $\pi_B^*D = \mu^*D_A$ for some $\Q$-Cartier divisor $D_A$. Now
$$D = \pi_B^*D|_{B \times \{0\}} = \mu^*(D_A|_B) = (D_A)|_B.$$

For (2), by assumption, we have an isogeny $h: E_1\times\cdots \times E_n \to A$. Now we consider the composition morphism (also a group homomorphism) $$g: E_1\times\cdots \times E_n \to A \to A/B.$$ Then there exist some indices $\{k_1,\cdots,k_{m}\}$ such that $$g|_{E_{k_1}\times \cdots E_{k_m}}: E_{k_1}\times \cdots \times E_{k_m} \to A/B$$ is finite and surjective. Hence $A/B$ is isogenous to the product of some elliptic curves. In particular, we can find a chain of subtori inside $A/B$ such that $$\{0\}=A_m'\subset A_{m-1}'\cdots\subset A_0'=A/B$$ with $\mathrm{codim}(A_i')=i$ for $0\leq i\leq m$, and this induces a chain of subtori $$B= A_m\subset \cdots \subset A_0=A$$ with $\mathrm{codim}(A_i)=i$ for $0\leq i\leq m$.
\end{proof}

\begin{lemma}\label{lem: maximal Albanese dimension}
    Let $X$ be a normal projective variety such that $X$ admits a polarized endomorphism and has a generically finite morphism $\alpha: X \to A$ into an abelian variety $A$ (or equivalently $X$ has maximal Albanese dimension).
    Then both $X$ and $\alpha(X)$ are abelian varieties, $\alb_X: X \to \Alb(X)$ is an isomorphism, and $\alpha$ factors as $\alb_X$ and a finite morphism (into $A$) $\Alb(X) \to A$.
\end{lemma}

\begin{proof}
By \cite[Theorem~1.3]{NZ}, \cite[Corollary~3.5]{Mor87} and the ramification divisor formula for $\alpha$, the Kodaira dimensions satisfy $0 \ge \kappa(X) \ge \kappa(\alpha(X)) \ge 0$. Hence $\kappa(X) = \kappa(\alpha(X)) =0$, and $\alpha(X)$ is a translation of a subtorus of $A$. In particular, $X$ is non-uniruled. It follows that $X$ is $Q$-abelian: there exists a quasi-\'etale finite surjective morphism $A' \to X$ from an abelian variety $A'$ (cf.~\cite[Theorem~3.4]{NZ}, \cite[Theorem~1.21]{GKP13}). 

The composition morphism $A' \to X \to \alpha(X)$ is generically finite and hence finite. Since $A' \to X$ is surjective, it follows that $\alpha: X \to \alpha(X)$ is also a finite morphism. By the universal property of $\Alb(X)$, the morphism $\alpha: X \to \alpha(X)$ factors 
through the Albanese morphism: $$X \xrightarrow{\alb_X} \Alb(X) \xrightarrow{\tau} \alpha(X).$$
By \cite[Theorem~1]{Kaw81} and noting that $X$ has quotient singularities (hence rational singularities), $\alb_X$ is surjective with connected fibers. Since $\alb_X$ is finite, it is an isomorphism.

\end{proof}

\section{Exponents and ranks of kernels of polarized isogenies: Proof of Theorem~\ref{thm: rank main}}\label{section 3}

In this section, we study the exponents and ranks of the kernels of polarized self-isogenies of an abelian variety.

\medskip

We recall some facts about polarizations of abelian varieties. Let $\mathcal{L}$ be an ample line bundle on an abelian variety $A$ of dimension $d$. Then $\mathcal{L}$ determines a polarization $\lambda_{\mathcal{L}}: A\to A^\vee$, where $\lambda_{\mathcal{L}}$ is an isogeny from $A$ to its dual $A^\vee$ with a finite kernel $K(\mathcal{L}):=\Ker(\lambda_{\mathcal{L}})$. Moreover, we have $|K(\mathcal{L})|=\chi(\mathcal{L})^2$, and by the Riemann-Roch formula, $\chi(\mathcal{L})=\frac{1}{d!}\mathrm{Vol}(\mathcal{L}) = \frac{1}{d!}(\mathcal{L}^d)$.

\begin{lemma}\label{lem: rank main 1}
Let $f$ be a $q$-polarized self-isogeny of an abelian variety $A$ and let $\mathcal{L}$ be an ample line bundle on $A$ that is numerically $f$-periodic. Then for any $f$-stable subtorus $B$ of $A$, we have the following.
\begin{enumerate}
    \item $\rank \Ker (f^s|_B)\geq \dim(B)$ for all integers $s$ with $$s> (\dim(A)-1)\cdot \log_q |K(\mathcal{L})|.$$
    \item For any prime number $p\mid q$, we have $\rank_p \Ker (f^s|_B)\geq \dim(B)$ for all integers $$s>(\dim(A)-1)\cdot \log_p |K(\mathcal{L})|.$$
\end{enumerate}
Moreover, if $f^*\mathcal{L}\equiv\mathcal{L}^{\otimes q}$, then for each positive integer $s$, $$\exp(\Ker (f^s))\ \big|\   q^s|K(\mathcal{L})|.$$
\end{lemma}

\begin{proof}
    Let $s$ be a positive integer. Since $B$ is $f$-stable, $f^s|_B$ is $q^s$-polarized. Hence $|\Ker(f^s|_B)|=q^{\dim(B)\cdot s}$.

    \medskip

    We consider the polarization $$\lambda_{\mathcal{L}}: A \to A^\vee=\mathrm{Pic}^0(A)\ \ \ \ \ \ \lambda_{\mathcal{L}}(x)=t_x^*\mathcal{L}\otimes \mathcal{L}^{-1}$$ induced by $\mathcal{L}$, where $t_x$ denotes the translation-by-$x$ map on $A$. Let $K(\mathcal{L})$ be the kernel of $\lambda_{\mathcal{L}}$. Since $\mathcal{L}$ is ample, $|K(\mathcal{L})|$ is a finite integer and $$|K(\mathcal{L})|=\chi(\mathcal{L})^2=(\frac{1}{\dim(A)!}\cdot\mathrm{Vol}(\mathcal{L}))^2.$$ By \cite[Page~74,\S II.8]{Mum74}, $\lambda_{\mathcal{L}}$ is defined up to algebraic equivalence of $\mathcal{L}$, hence defined on $\NS(A)$.

    For simplicity, we denote $f^n$ by $f_n$ for each positive integer $n$. Since $\mathcal{L}$ is numerically $f$-periodic, there exists an integer $m > 0$ such that $f_m^*\mathcal{L}$ is numerically parallel to $\mathcal{L}$. Since $f$ is $q$-polarized, $f_m^*\mathcal{L}\equiv \mathcal{L}^{q^m}$. In particular, $f_m^*\mathcal{L}\otimes \mathcal{L}^{-q^m}$ (mod $\Pic^0(A)$) is a torsion element in $\NS(A)$ (cf. \cite[Corollary~1.4.38]{Laz}). Since $\NS(A)$ is torsion free, this further implies that $f_m^*\mathcal{L}$ and $\mathcal{L}^{q^m}$ represent the same class in $\NS(A)$. In particular, we have $$\lambda_{f_m^*\mathcal{L}}=\lambda_{\mathcal{L}^{q^m}}=q^m\cdot \lambda_\mathcal{L}.$$

    Now we replace $f_m$ with its iteration $f_m^{s'}=f_{ms'}$.  Let $f_{ms'}^\vee$ be the dual of $f_{ms'}$ induced via pulling back line bundles on $A$. Note that the following diagram commutes (cf. \cite[\S III.15, Theorem~1]{Mum74}). $$\begin{tikzcd} A \arrow[r, "f_{ms'}"] \arrow[d, "\lambda_{f_{ms'}^*\mathcal{L}}"'] & A \arrow[d, "\lambda_{\mathcal{L}}"] \\ A^\vee  & A^\vee \arrow[l, "f_{ms'}^\vee"'] \end{tikzcd}$$
    Hence, for any point $x\in A$, we have
    \begin{align}\label{eq: the duality composition relation}
        f_{ms'}^\vee\circ\lambda_\mathcal{L}\circ f_{ms'}(x)=\lambda_{f_{ms'}^*\mathcal{L}}(x)=\lambda_{\mathcal{L}^{q^{ms'}}}(x)=q^{ms'}\cdot \lambda_{\mathcal{L}}(x),
    \end{align}
    where $q^{ms'}\cdot \lambda_{\mathcal{L}}(x)$ means the multiplication-by-$q^{ms'}$ on the abelian variety $A^\vee$. Thus, the LHS (and hence the RHS) of \eqref{eq: the duality composition relation} is trivial for all $x \in \Ker(f_{ms'})$. So there is an induced group homomorphism $$\lambda_{\mathcal{L}}: \Ker(f_{ms'}) \to A^\vee [q^{ms'}],$$ where $A^\vee[q^{ms'}]$ denotes the $q^{ms'}$-torsion points on $A^\vee$. Hence, the exponent of the finite abelian group $\Ker(f_{ms'})$ satisfies
    \begin{align}\label{eq: estimation of the exponent of kernel}
        \mathrm{exp}(\Ker(f_{ms'}))\ \big| \  q^{ms'}|K(\mathcal{L})|.
    \end{align}
    As a consequence, we have $\mathrm{exp}(\Ker(f_{ms'}|_B))\ \big| \  q^{ms'}|K(\mathcal{L})|$. Set $r_{ms'}:= \rank \Ker(f_{ms'}|_B)$. Then by the structure theorem of finite abelian groups, we have
    $$q^{ms'\cdot \dim(B)}= \deg f_{ms'}|_B = |\Ker(f_{ms'}|_B)|\leq (q^{ms'}|K(\mathcal{L})|)^{r_{ms'}}.$$ This implies that 
    \begin{align}\label{eq: global estimation of global rank}
        r_{ms'}\geq \frac{ms'}{ms'+\log_q |K(\mathcal{L})|}\dim(B)
    \end{align}
    Since $r_{ms'}$ is an integer and $\dim(A)\geq \dim(B)$, by choosing $s=ms'$ such that $$s> (\dim(A)-1)\cdot \log_q|K(\mathcal{L})|,$$ we have $r_s=r_{ms'}\geq \dim(B)$. This proves (1).

    \medskip
    
    Let $\Ker_p(f_{ms'}|_B)$ be the Sylow $p$-subgroup of $\Ker(f_{ms'}|_B)$, and write $q=p^v\cdot u$ for some positive integers $v,u$ such that $\mathrm{gcd}(p,u)=1$. Then \eqref{eq: estimation of the exponent of kernel} implies that
    \begin{align}\label{eq: estimation of the exponent of p kernel}
        \mathrm{exp}(\Ker_p(f_{ms'}|_B))\ \big| \  p^{vms'}|K(\mathcal{L})|.
    \end{align}
    Suppose that $\rank_p \Ker (f_{ms'}|_B)=\rank \Ker_p(f_{ms'}|_B)=r_{ms',p}$ for all positive integers $s'$. Since $$|\Ker(f_{ms'}|_B)|=\deg(f_{ms'}|_B)=q^{\dim(B)\cdot ms'},$$ we have $|\Ker_p(f_{ms'}|_B)|=p^{\dim(B)\cdot vms'}$. By the structure theorem of finite abelian groups, we have $$p^{\dim(B)\cdot vms'}=|\Ker_p(f_{ms'}|_B)|\leq (p^{vms'}|K(\mathcal{L})|)^{r_{ms',p}},$$ hence
    $$r_{ms',p}\geq \frac{vms'}{vms'+\log_p|K(\mathcal{L})|} \dim(B).$$ Since $r_{ms',p}$ is an integer, by choosing $s=ms'$ such that $s> \frac{1}{v} (\dim(A)-1)\cdot \log_p|K(\mathcal{L})|$, we have $r_{s,p}=r_{ms',p}\geq \dim(B)$. This proves (2).
\end{proof}

The following, which includes Theorem~\ref{thm: rank main}, follows as a direct corollary of Lemma~\ref{lem: rank main 1}.

\begin{theorem}\label{thm: rank main 1}
Let $f$ be a $q$-polarized self-isogeny of an abelian variety $A$. Then for any prime $p\mid q$, we have $\rank_p \Ker (f^s)\geq \dim(A)$ for some positive integer $s$. In particular, $\rank \Ker (f^s)\geq \dim(A)$ for some positive integer $s$.
\end{theorem}

As another corollary of Lemma~\ref{lem: rank main 1}, when $\rho(A)\leq 2$, we have the following uniform bound for the integer $s$ such that $\rank \Ker(f^s)\geq \dim(A)$.

\begin{proposition}\label{prop: rank main 1 uniform}
Let $A$ be an abelian variety. Then we may choose an integer $s$ depending only on $A$ that satisfies the following.

Let $f$ be a $q$-polarized self-isogeny of $A$. Assume either one of the following conditions.
\begin{enumerate}
    \item Every element in $\NS(A)$ is numerically $f$-periodic.
    \item $(f^k)^*$ acts as a scalar multiplication map on $\NS(A)$ for some integer $k \ge 1$.
    \item $\rho(A)\leq 2$.
\end{enumerate}
Then for any prime $p\mid q$, we have $\rank_p \Ker (f^s)\geq \dim(A)$. In particular, $\rank \Ker (f^s)\geq \dim(A)$.
\end{proposition}

\begin{proof}
    We fix an ample line bundle $\mathcal{L}$ on $A$. Let $f$ be a $q$-polarized self-isogeny of $A$ such that the pair $(A,f)$ satisfies one of our conditions above. Then $\mathcal{L}$ is numerically $f$-periodic (we remark that when $\rho(A)\leq 2$, $(f^2)^*$ acts as a scalar multiplication on the N\'eron-Severi group $\NS(A)$). By Lemma~\ref{lem: rank main 1}, we have $$\rank_p \Ker (f^s)\geq \dim(A)$$ for all $s>\log_p|K(\mathcal{L})|\cdot (\dim(A)-1)$. This completes the proof.
\end{proof}

In Lemma~\ref{lem: rank main 1}, we can take $s=1$ if $A$ is principally polarized by a numerically $f$-periodic line bundle. Hence, we obtain the following.

\begin{corollary}\label{cor: explicit index s principal}
    Let $A$ be an abelian variety and let $f$ be a $q$-polarized self-isogeny of $A$. Suppose that $A$ is principally polarized by a numerically $f$-periodic line bundle $\mathcal{L}$. Then for any prime $p\mid q$, we have $\rank_p \Ker (f)\geq \dim(A)$. In particular, $\rank \Ker (f)\geq \dim(A)$.
\end{corollary}

\begin{proof}
    Notice that $|K(\mathcal{L})|=1$. The assertion follows directly from Lemma~\ref{lem: rank main 1}.
\end{proof}

The following example shows that the rank bound in Lemma~\ref{lem: rank main 1}, and hence in Theorem~\ref{thm: rank main 1} and Proposition~\ref{prop: rank main 1 uniform}, is optimal.

\begin{example}\label{ex: optimal rank bound for rank main}
    Let $E=\mathbb{C}/\mathbb{Z}[i]$ be an elliptic curve, and let $A=E\times\cdots \times E$ be the product of $n$ copies of $E$. We consider the map $f$ given by $\mathrm{diag}(4+3i,\cdots, 4+3i)$. Then $\Ker(f^s)=(\mathbb{Z}/25^s\mathbb{Z})^n$ by Remark~\ref{rem: simple lemma to check kernel is cyclic}. Hence $\rank \Ker(f^s)=n=\dim(A)$ for all $s$.
\end{example}

\section{Essential dimensions of polarized self-isogenies}\label{sect: isogenies}

In this section, we study the essential dimensions of polarized self-isogenies of abelian varieties.

\subsection{Proof of {Theorem~\ref{thm: ed main} and Proposition~\ref{thm: quasi-abelian and non-uniruled}}}\label{section 4.1}

In this subsection, we discuss the essential dimension for a polarized self-isogeny $f$ of an abelian variety whose subtori are all $f$-periodic.

Recall that Prokhorov and Shramov proved the Jordan property for birational automorphism groups of rationally connected varieties of fixed dimension (\cite[Theorem~1.8 and Theorem~1.10]{PS14}) assuming the Borisov–Alexeev–Borisov conjecture (which is now a theorem of Birkar \cite{Bir21}). We recall one of their byproducts in the proof of \cite[Theorem~1.10]{PS14}, which will later play an important role in the proof of Theorem~\ref{thm: main with all subtori f numerically periodic} (hence Theorem~\ref{thm: ed main}). For the reader's convenience, we also give a proof here.

\begin{lemma}[{\cite[Theorem~1.10]{PS14}}]\label{lem: uniform finite index embedding to general linear group}
    Fix dimension $n$. Then there exists a constant $J=J(n)$ depending only on $n$ satisfying the following.
    For any rationally connected variety $X$ of dimension $n$ and any finite subgroup $G\subset \mathrm{Bir}(X)$, there exists an abelian subgroup $H\leq G$ with an embedding $H\leq \mathrm{GL}(n,\mathbb{C})$ such that $[G:H]\leq J(n)$. In particular, $\rank H\leq n$.
\end{lemma}

\begin{proof}
    Let $X$ be a rationally connected variety of dimension $n$, and let $G \leq \mathrm{Bir}(X)$ be a finite group. Let $\widetilde{X}$ be a regularization of $G$, i.e. $\widetilde{X}$ is a projective variety with an action of $G$ and a $G$-equivariant birational map $\xi: \widetilde{X} \dashrightarrow X$ (see \cite[Theorem~1.4]{deFE02}). Taking a $G$-equivariant resolution of singularities (\cite[Theorem~13.2]{BM97}), one can assume that $\widetilde{X}$ is smooth. Note that $\widetilde{X}$ is rationally connected since so is $X$. By \cite[Theorem~4.2]{PS14} there is a constant $J_1$ depending only on $n$ such that there exists a subgroup $H' \subset G$ of index at most $J_1$ and a point $x \in \widetilde{X}$ fixed by $H'$. The action of $H'$ on the Zariski tangent space $T_x(\widetilde{X})$ is faithful (cf. \cite[Lemma~2.2]{KS13}). Hence, we have an embedding $H'\leq \mathrm{GL}(n,\mathbb{C})$. By the classical Jordan property of $\mathrm{GL}(n,\mathbb{C})$, there is a constant $J_2$ depending only on $n$ such that $H'$ has an abelian subgroup $H$ of index at most $J_2$. We may take $J=J_1J_2$. This completes the proof.
\end{proof}

Let $J$ be the constant in Lemma~\ref{lem: uniform finite index embedding to general linear group}. We give a sufficient condition for incompressibility.

\begin{lemma}\label{lem: main lemma with all subtori f stable}
    Let $A$ be an abelian variety and $f$ a $q$-polarized self-isogeny of $A$. Suppose that
    \begin{enumerate}
        \item every subtorus of $A$ is $f$-stable, and
        \item there exists a constant $C\in \Z_{>0}$ such that the exponent $\mathrm{exp}(\mathrm{Ker}(f^k))$ divides $Cq^k$ for each positive integer $k$.
    \end{enumerate}
    Then $f^s$ is incompressible for all positive integers $s>\log_q(J)+\dim(A)\cdot \log_q C$.
\end{lemma}

\begin{proof}
    Fix a positive integer $s$. By our assumption, we have a birational base change diagram: 
    $$\xymatrix{
    A \ar[r]^{f^s} \ar@{-->}[d] & A \ar@{-->}[d] \\
    Y_s \ar[r]^{h_s} & Y_s'
    }$$
    where the horizontal maps are finite and $\dim (Y_s)=\dim (Y_s')=\mathrm{ed}(f^s)\leq \dim(A)$. We may assume that
    \begin{enumerate}
        \item $Y_s, Y_s'$ are normal (or smooth by \cite[Theorem~13.2]{BM97}),
        \item the rational maps $A \dashrightarrow Y_s'$ and $A \dashrightarrow Y_s$ have connected general fibers, and
        \item $\Ker(f^s)$ acts birationally and faithfully on $Y_s$.
    \end{enumerate}
    Let $g_s: Y_s\to \mathrm{Alb}(Y_s)$ be the Albanese map. By \cite[Corollary~7]{KZ}, $g_s$ is surjective with rationally connected general fibers, and we have an induced surjective morphism of abelian varieties $h_s: A\to \mathrm{Alb}(Y_s)$. Let $B_s:=\Ker (h_s)$. We set $$G_s:=B_s\cap \Ker(f^s)=\Ker(f^s|_{B_s}).$$
    Then $G_s$ acts birationally and faithfully on a general fiber $F_s$ of $g_s$. Since $B_s$ is $f^s$-stable, $f^s|_{B_s}$ is $q^s$-polarized, hence $$|G_s|= \deg(f^s|_{B_s}) = q^{\dim(B_s)\cdot s}.$$ By Lemma~\ref{lem: uniform finite index embedding to general linear group}, there exists a constant $J$ depending only on $n$ and an abelian subgroup $H_s\leq G_s$ such that $[G_s :H_s]\leq J$ and $H_s\leq \mathrm{GL}(\dim(F_s),\mathbb{C})$. In particular, $$\rank(H_s)\leq \dim(F_s)=\dim(B_s)-(\dim(A)-\dim(Y_s)) \leq \dim (B_s).$$ By our construction, the exponent of $H_s$ satisfies $\exp(H_s)\leq \exp(G_s)\leq Cq^s$. Hence $$q^{\dim(B_s)\cdot s}=|G_s| \le J\cdot |H_s|\leq J\cdot (Cq^s)^{\dim(F_s)}.$$ It follows that 
    \begin{align}\label{eq: control of the essential dimension and dimension}
        q^{(\dim(A)-\mathrm{ed}(f^s))\cdot s}=q^{(\dim(B_s)-\dim(F_s))\cdot s}\leq J\cdot C^{\dim(F_s)}\leq J\cdot C^{\dim(A)}.
    \end{align}
    Now we take $s>\log_q(J)+\dim(A)\cdot \log_q C$. If $\dim(A)>\mathrm{ed}(f^s)$, then by \eqref{eq: control of the essential dimension and dimension} we have $$(J\cdot C^{\dim(A)})^{(\dim(A)-\mathrm{ed}(f^s))}<q^{(\dim(A)-\mathrm{ed}(f^s))\cdot s}\leq J\cdot C^{\dim(A)},$$ a contradiction. It follows that $\dim(A)=\mathrm{ed}(f^s)$.
\end{proof}

\begin{remark}\label{rem: dim two case choose the constant J}
    Note that in Lemma~\ref{lem: main lemma with all subtori f stable}, if $\dim(A)=2$, then $\dim(F_s)\leq 1$. The only smooth rationally connected variety of dimension one is $\mathbf{P}^1$. By the classification of finite subgroups of $\mathrm{Aut}(\mathbf{P}^1)=\mathrm{PGL}(2)$ (cf. \cite{Bea10}), finite abelian subgroups of $\mathrm{PGL}(2)$ are either cyclic or equal to the Klein group. Hence, we can take $J=2$ in this case.
\end{remark}

Now we are ready to prove our main theorem below, which implies Theorem~\ref{thm: ed main}.

\begin{theorem}\label{thm: main with all subtori f numerically periodic}
        Let $A$ be an abelian variety and $f$ a polarized endomorphism of $A$. If every subtorus of $A$ is numerically $f$-preperiodic, then $f^s$ is incompressible for some positive integer $s$.
\end{theorem}

\begin{proof}
    By replacing $f$ with its translation, we may assume that $f$ is a self-isogeny. In fact, for any positive integer $s$ and any $x\in A$, we have $(f\circ t_x)^s=f^s\circ t_{x_s}$ for some $x_s\in A$. Hence $\mathrm{ed}((f\circ t_x)^s)=\mathrm{ed}(f^s)$. By Lemma~\ref{lem: uniform stable index}(2), we may replace $f$ with an iterate so that every subtorus of $A$ is $f$-stable.

    Since $f$ is polarized, there exists an ample line bundle $\mathcal{L}$ on $A$ and a positive integer $q>1$ such that $f^*\mathcal{L}\equiv \mathcal{L}^{\otimes q}$. By Lemma~\ref{lem: rank main 1}, we have 
    \begin{align}\label{eq: control of the exponent}
        \exp(\Ker (f^k))\ \big|\   q^k|K(\mathcal{L})|
    \end{align}
    for each positive integer $k$. The assertion follows from Lemma~\ref{lem: main lemma with all subtori f stable} by taking $C=|K(\mathcal{L})|$.
\end{proof}

\begin{proof}[Proof of Proposition~\ref{thm: quasi-abelian and non-uniruled}]
Since $X$ is non-uniruled, $X$ is $Q$-abelian (cf. \cite[Theorem~3.4]{NZ} and \cite[Theorem~1.21]{GKP13}). Now let $A \to X$ be the Albanese closure in codimension one as defined in \cite[Lemma 2.12]{NZ}. Then the proposition follows from \cite[Proposition 5.4, Step 2]{NZ}. Indeed, for the last part, for $f_A^s: A = A_1 \to A = A_2$, the $A_1$ is the normalization of the fibre product of $f^s$ and $\sigma: A_2 \to X$. This completes the proof.
\end{proof}

\subsection{Isogenous to products of elliptic curves case: Proof of Theorem~\ref{thm: essential dimension isogenous to product elliptic curves intro}}

In this subsection, we focus on essential dimensions of polarized endomorphisms when $A$ is isogenous to products of elliptic curves. We start with a lemma on the structures of subtori and polarized self-isogenies of $A$.

\begin{lemma}\label{lem: basic property of polarized end of product abelian}
  Let $A$ be an abelian variety that is isogenous to the product of elliptic curves. Let $f$ be a $q$-polarized self-isogeny of $A$. Then the following are equivalent.
  \begin{enumerate}
      \item $f^*|_{\NS(A)_{\mathbb{Q}}}$ is a scalar multiplication.
      \item Every subtorus of $A$ is $f$-stable.
      \item Every one-dimensional subtorus of $A$ is $f$-stable.
  \end{enumerate}
  Under the above equivalent conditions, we have the exponent $\mathrm{exp}(\Ker(f))$ divides $q$ and $$\rank_p \Ker(f|_B)\geq \dim(B)$$ for any prime $p\mid q$ and any subtorus $B$ of $A$.

  Moreover, if every one-dimensional subtorus of $A$ is numerically $f$-preperiodic, then there exists an integer $s$ depending on $f$ such that every subtorus of $A$ is $f^s$-stable.
\end{lemma}

\begin{proof}
First we prove the equivalence of (1), (2) and (3).

\medskip

We first assume (1) and prove (2). We have every codimension-one subtorus is $f$-stable; see Lemma~\ref{lem: numerically parallel abelian varieties}. Let $B$ be a subtorus of codimension $m < n$. We claim that $B$ is $f$-stable. Indeed, by Lemma~\ref{lem: basic property of abelian Neron Severi}(2), we have a sequence
$B = A_m \subset A_{m-1} \subset \cdots \subset A_1$ of subtori such that $\mathrm{codim}(A_k)=k$ for each $k$.
By Lemma~\ref{lem: basic property of abelian Neron Severi}(1), $\NS(A_k)_{\Q} = (\NS(A)_{\Q})|_{A_k}$ for all $k$, and $A_1$ is $f$-stable, $(f|_{A_1})^*|_{\NS(A_1)_{\Q}} =q \id$. Hence $A_2$ is $f$-stable. Inductively, $B$ is $f$-stable.

Clearly (2) implies (3).  We now assume (3) and prove (1). By our assumption, we have an isogeny $h: E_1^{n_1}\times\cdots\times E_k^{n_k} \to A$, where $n_i$ are positive integers such that $\sum_{i=1}^k n_i=n$, $E_i$ and $E_j$ are not isogenous for $i\neq j$, and $E_i^{n_i}$ represents the self-product $n_i$ times. We may write each elliptic curve $E_i$ as $\mathbb{C}/\Lambda_i$ and $A$ as $\mathbb{C}^n/\Lambda$. Note that $h$ is represented by a matrix multiplication $$h(z_1,\cdots,z_n)=(z_1,\cdots,z_n)\cdot M_h$$ between the corresponding universal covering spaces. Note that $$M_h\in \mathrm{End}(E_1^{n_1}\times\cdots\times E_k^{n_k})\otimes\mathbb{Q}\subseteq\mathrm{diag}(M_{n_1\times n_1}(\mathbb{C}),\cdots, M_{n_k\times n_k}(\mathbb{C})).$$ Therefore, by reorganizing the order of the coordinates for $A$, we may write $$M_h=\mathrm{diag}(M_{n_1},\cdots,M_{n_k})$$ for $M_{n_i}\in M_{n_i\times n_i}(\mathbb{C})$.

Now we return to the self-isogeny $f$ of $A$. We assume that $f$ is represented by a matrix $M_f\in M_{n\times n}(\mathbb{C})$. We consider the one-dimensional subtorus $$T_i:= (0,\cdots,0,z_i,0,\cdots,0)\cdot M_h$$ for each $1\leq i\leq n$, where $z_i$ denotes the $i$-th coordinate of $E_1^{n_1}\times\cdots\times E_k^{n_k}$. And we further consider the one-dimensional subtorus $$T_{i,j}= (0,\cdots,z_i,\cdots,z_j,\cdots,0)\cdot M_h$$ for each pair of indices $i,j$ with $z_i=z_j=:z$, where elliptic curves corresponding to the $i,j$ coordinates are isogenous. By our assumption, $T_i$ and $T_{i,j}$ are $f$-stable for all $i,j$. Since $T_i$ is $f$-stable, we have $$M_h\cdot M_{f}= \diag(\lambda_1',\dots,\lambda_n')\cdot M_h,$$ where $\lambda_i'$ are complex numbers such that $|\lambda_i'|^2=q$. Since $T_{i,j}$ are $f$-stable, 
\begin{align*}
    (0,\cdots,\lambda_i' z_i,\cdots,\lambda_j' z_j,\cdots,0)\cdot M_h &=f((0,\cdots,z_i,\cdots,z_j,\cdots,0)\cdot M_h) \\&=(0,\cdots,\lambda_{i,j}'z_i,\cdots,\lambda_{i,j}'z_j,\cdots,0)\cdot M_h
\end{align*}
for some complex number $\lambda_{i,j}'$. It follows that $\lambda_i'=\lambda_j'$. In particular, we have $$M_h\cdot M_{f} = \mathrm{diag}(\lambda_1 I_1,\cdots,\lambda_k I_k)\cdot M_h$$ for some complex numbers $\lambda_1,\dots,\lambda_k$ whose norms are $q$. Here $I_i$ denotes the $n_i\times n_i$ identity matrix for each $1\leq i \leq k$. It follows that $$M_f=M_h^{-1}\cdot \mathrm{diag}(\lambda_1 I_1,\cdots,\lambda_k I_k) \cdot M_h=\mathrm{diag}(\lambda_1 I_1,\cdots,\lambda_k I_k).$$ By \cite[Theorem~4.3]{PreS12}, $f^*|_{\mathrm{NS}_{\mathbb{Q}}}$ is a scalar multiplication map. This completes the proof of the equivalence of statements (1)-(3).

\medskip

Next, we prove the `Moreover' part. By Lemma~\ref{lem: uniform stable index}(1), by replacing $f$ with its iterate $f^s$ for some positive integer $s$, we may assume that all $T_i$ and $T_{i,j}$ defined in the previous argument are $f^s$-stable (see Lemma~\ref{lem: numerically parallel abelian varieties}). Now, the same proof implies that $(f^s)^*|_{\mathrm{NS}_{\mathbb{Q}}}$ is a scalar multiplication map. This completes the proof of the `Moreover' part.

\medskip

Now we assume the equivalent conditions (1)-(3). Since the multiplication-by-$\lambda_i$ defines a morphism of $E_i$, by Lemma~\ref{lem: complex structure of endomorphism ring of abelian variety}, $\lambda_i+\bar \lambda_i=l_i\in \mathbb{Z}$. In particular, the map $\bar{f}$ represented by the matrix $$\mathrm{diag}(\bar\lambda_1 I_1,\cdots,\bar\lambda_k I_k)=\mathrm{diag}( l_1 I_1,\cdots, l_k I_k)-\mathrm{diag}(\lambda_1 I_1,\cdots,\lambda_k I_k)$$ induces a self-isogeny of $A$ as it preserves the lattice of $A$. Now, for each element $x\in \Ker(f)$, $0 = \bar f\circ f(x)=q\cdot x$. Hence $\mathrm{exp}(\Ker(f))\mid q$. For the last part, we take a prime $p\mid q$. Then we may write $q=p^v \cdot u$, where $u,v$ are positive integers such that $\mathrm{gcd}(u,p)=1$. Let $B$ be a non-trivial subtorus of $A$ (which is $f$-stable), and let $r:=\rank\Ker_p(f|_B)$. Since $f|_B$ is $q$-polarized, $\deg(f|_B)=q^{\dim(B)}=|\Ker(f|_B)|$. Let $\Ker_p (f|_B)$ be the Sylow $p$-subgroup of $\Ker(f|_B)$. Then $\mathrm{exp}(\Ker_p(f|_B))\mid p^v$. We have $$p^{v\cdot \dim(B)}=|\Ker_p (f|_B)|\leq p^{v\cdot r}.$$ It follows that $r\geq \dim(B)$.
\end{proof}

As a corollary of the above lemma, we have the following, which includes Theorem~\ref{thm: essential dimension isogenous to product elliptic curves intro} as a special case.

\begin{corollary}\label{cor: essential dimension isogenous to product elliptic curves periodic}
    Let $A$ be an abelian variety that is isogenous to the product of elliptic curves. Let $f$ be a $q$-polarized endomorphism of $A$. Then the following statements hold.
    \begin{enumerate}
        \item If every one-dimensional subtorus in $A$ is numerically $f$-preperiodic, then $f^s$ is incompressible for some positive integer $s$.
        \item If every one-dimensional subtorus in $A$ is numerically $f$-stable, then we can choose the index $s$ in (1) so that it depends only on $\dim(A)$.
        \item Under the assumption in (2), for any prime integer $p\mid q$, we have $$\mathrm{ed}(f)\geq \frac{p-1}{p} \dim(A).$$
    \end{enumerate}

\end{corollary}

\begin{proof}
    As in Theorem~\ref{thm: main with all subtori f numerically periodic}, by replacing $f$ with its translation, we may assume that $f$ is a self-isogeny.
    
    For (1), by Lemma~\ref{lem: basic property of polarized end of product abelian}, there is a positive integer $s$ such that every subtorus $B$ of $A$ is $f^s$-stable. Now (1) follows from Theorem~\ref{thm: main with all subtori f numerically periodic}.

    For (2), by Lemma~\ref{lem: basic property of polarized end of product abelian}, every subtorus $B$ of $A$ is $f$-stable and the exponent of $\Ker f$ divides $q$. The assertion now follows from Lemma~\ref{lem: main lemma with all subtori f stable} by taking $C=1$.

    For (3), by Lemma~\ref{lem: basic property of polarized end of product abelian}, we have $\rank_p \Ker(f|_B)\geq \dim(B)$ for every subtorus $B\subset A$ of dimension $m\geq 1$ and any prime $p\mid \deg(f)$. By \cite[Theorem~2]{KZ}, we have $$\mathrm{ed}(f) \geq \dim(A) - \dim(B) + \frac{p-1}{p} \rank_p(\Ker (f) \cap B)\geq\dim(A)-\frac{1}{p}\dim(B)$$ for some subtorus $B$ of $A$ and all prime integers $p \, | \, q$. It follows that $$\mathrm{ed}(f) \geq \frac{p-1}{p}\dim(A).$$ This completes the proof.
\end{proof}
\section{Abelian surfaces}

\subsection{Counterexamples}

First, we introduce the following criterion for a self-isogeny of an abelian surface to be compressible.

\begin{proposition}\label{prop: essential dimension be one criterion}
    Let $f: A \to A$ be a polarized endomorphism of an abelian surface $A$. Suppose that there exists a sub-1-torus $E$ in $A$ such that the restriction $f|_E: E\to f(E)$ has degree one. Then $\mathrm{ed}(f)\leq 1$.
\end{proposition}

\begin{proof}
    By Lemma~\ref{lem: numerically f-periodic elliptic curve and quotient}(2), we have an induced surjective isogeny $g: Y_1:=A/E \to Y_2:=A/f(E)$ and the following commutative diagram,
$$\begin{tikzcd}
X \arrow[rrd, bend left, "h"] \arrow[rdd, bend right] & & \\
 & A \arrow[lu, dashed, "\sigma"] \arrow[r, "f"] \arrow[d, "\tau_1"'] & A  \arrow[d, "\tau_2"] \\
& Y_1 \arrow[r, "g"] & Y_2
\end{tikzcd}$$
where $\tau_1,\tau_2$ are the natural quotient morphisms. Moreover, we have $\deg(g)=\deg(f)$. Let $X$ be the normalization of the fibre product $A\times_{Y_2} Y_1$ with the projection map $h: X\to A$. Since $\tau_2: A \to Y_2$ has connected fiber, $X$ is irreducible. By the universal property of fibre products, $f$ factors as $h \circ \sigma: A \to X \to A$. 
Now $\deg(h)=\deg(g)=\deg(f)$. So $\sigma$ is birational and finite and hence an isomorphism.
Hence $f$ is a base change of $g$. In particular, $\mathrm{ed}(f)\leq 1$.
\end{proof}

As a corollary, we introduce a way to construct 
compressible endomorphisms.

\begin{lemma}\label{lem: essential dimension integral matrix and factorization}
    Let $E$ be an elliptic curve, $A=E\times E$, and $f_M$ a self-isogeny of $A$ defined by $$f_M: A\to A,\ \ \ f_M(x,y)=(ax+by,cx+dy)=(x,y)\cdot M,$$ where $a,b,c,d$ are integers that represent the multiplication-by-integer maps on $E$. Then the following statements hold:
    \begin{enumerate}
        \item Suppose that there exists an integer row vector $V$ such that the greatest common divisor of the entries of $V\cdot M$ is one. Then $\mathrm{ed}(f_M)<2$.
        \item $\mathrm{ed}(f_M)$ equals $2$ if and only if $f_M$ factors through a multiplication-by-$m$ map $m_A$ on $A$ for some integer $m>1$.
        \item We have $\mathrm{ed}(f_M)=\mathrm{ed}(f_{PMQ})$ for any $P,Q\in \mathrm{GL}(2,\mathbb{Z})$. In particular, $\mathrm{ed}(f_M)=1$ if and only if $M$ has Smith normal form $\mathrm{diag}(\mathrm{det}(M),1)$, or equivalently, the greatest common divisor of the entries of $M$ is one.
    \end{enumerate}
\end{lemma}

\begin{proof}
    For (1), we define the following one-dimensional subtorus in $A$: $$\Delta_{p,q}=\{(px,qx)\mid x\in E,~ p,q\in \mathbb{Z},~ \mathrm{gcd}(p,q)=1\}.$$ A simple calculation shows that $f_M(\Delta_{p,q})=\Delta_{p',q'}$, where $$(p',q')=\frac{1}{l}(p,q)\cdot M =\frac{1}{l}(ap+bq,cp+dq)$$ and $l=\mathrm{gcd}(ap+bq,cp+dq)$. Moreover, $f_M$ restricts to an isogeny $f_M|_{\Delta_{p,q}}: \Delta_{p,q} \to \Delta_{p',q'}$ with degree $l^2$. Now, write $V = (p,q)$. Then the entries of $V$ are automatically coprime by our assumption. It follows that $l=1$ by our assumption. This implies that $\deg(f_M|_{\Delta_{p,q}})=1$, hence $\mathrm{ed}(f_M)<2$ by Proposition~\ref{prop: essential dimension be one criterion}.

    For (2), by Lemma~\ref{lem: facts of essential dimension}(3) and \cite[Theorem~1]{KZ}, if $f_M$ factors through a multiplication-by-$m$ map $m_A$ for some $m>1$, then $\mathrm{ed}(f_M)=2$. Now suppose that $\mathrm{ed}(f_M)=2$. We prove that $f_M$ factors through the multiplication-by-$m$ map for some integer $m>1$. For simplicity, we introduce a map $$\phi: \mathbb{Z}^2 \to \mathbb{Z}\ \ \ (v_1,v_2)\to \mathrm{gcd}(v_1,v_2).$$ By $(1)$, we conclude that $\phi(W\cdot M)>1$ for all integer row vectors $W$. Note that $\phi(W\cdot MM')=\phi(V\cdot M)$ for any invertible integer matrix $M'$. Now we consider the Smith normal form of $M$: $M=PNQ$ for some invertible integer matrices $P,Q$ and $N = \diag(d_1, d_2)$ such that $d_1,d_2\in \mathbb{Z}$ and $d_2\mid d_1$. We take $V'=(1,1)\cdot P^{-1}$. Then $$1 < \phi(V'\cdot M)=\phi((1,1)P^{-1}PNQ)=\phi((1,1)\mathrm{diag}(d_1,d_2))=\mathrm{gcd}(d_1,d_2)=d_2.$$
    Hence $1<d_2=\mathrm{gcd}(d_1,d_2)\mid \mathrm{gcd}(a,b,c,d)$. This implies that $f_M$ factors through the multiplication-by-$d_2$ map.

  (3) follows directly from Lemma~\ref{lem: facts of essential dimension}(2),
    noting that $f_P$ and $f_Q$ are automorphisms.
\end{proof}

The following example shows that Lemma~\ref{lem: essential dimension integral matrix and factorization}(2) does not hold in general even after iteration.

\begin{example}\label{ex: iteration does not factor through multiplication}
    We consider the abelian surface $A:=E\times E$, where $E:=\mathbb{C}/\mathbb{Z}[i]$. For each $\alpha\in \mathbb{Z}[i]$, we consider the polarized self-isogeny of $A$ defined by $$f_{\alpha}: (z_1,z_2) \mapsto (\alpha\cdot z_1,\alpha\cdot z_2).$$ Then $\deg(f_\alpha)=|\alpha|^4$. In particular, $f^*$ and $f_*$ act as $|\alpha|^2\cdot \mathrm{id}$ on the N\'eron-Severi group of $A$. Hence every $1$-subtorus of $A$ is stable under $f_\alpha$. We claim that $\mathrm{ed}(f_\alpha)=2$ whenever $|\alpha|^2>2$. Suppose to the contrary that $f_\alpha$ is birationally a base change of a morphism $h: Y\to Y'$ for some smooth curves $Y$ and $Y'$. By \cite[Corollary~7]{KZ}, the map $Y\to \mathrm{Alb}(Y)$ is surjective.  $G : = \Ker (f_\alpha)$ descends to a faithful action on $Y$ so that $Y' = Y/G$. We set $B:=\Ker (A\to \mathrm{Alb}(Y))$. 
    
    Suppose that $\dim(B)=2$, i.e., $B=A$, then $Y=Y'=\mathbf{P}^1$. Since $\Ker (f_\alpha)\subset \mathrm{Aut}(\mathbf{P}^1)=\mathrm{PGL}(2)$ and $|\Ker (f_\alpha)|= |\alpha|^4> 4$, $\Ker (f_\alpha)$ is cyclic, which is a contradiction. 

    Suppose that $\dim(B) = 1$. Since $\deg(f_\alpha|_B) = |\alpha| > 1$, we have $\rank_p\Ker (f_\alpha|_B)\geq 1$ for any prime $p \, | \, \deg(f_\alpha|_B)$. Applying \cite[Theorem~2]{KZ} (see also its proof) for $f_\alpha$ and $B$: $$\mathrm{ed}(f_\alpha)\geq \dim(A)-\dim(B)+\frac{p-1}{p}\dim(B)=\dim(A)-\frac{1}{p}\dim(B) \ge 1.5.$$
    This proves that $\mathrm{ed}(f_\alpha)=2$.

    \medskip

    Now take $\alpha=4+3i$. Then $f_\alpha$ is polarized and has essential dimension two. However, $f_\alpha^s$ never factors through any $m_A$ with $m \ge 2$ (see Remark~\ref{rem: simple lemma to check kernel is cyclic}(2)) for any positive integer $s>0$.
\end{example}

\begin{remark}
    We remark that if $f$ is a surjective self-isogeny of $A$, where $A=\mathbb{C}/\mathbb{Z}[i] \times \mathbb{C}/\mathbb{Z}[i]$, such that $\mathrm{ed}(f)=2$, then $f$ factors through the isogeny $\alpha\cdot \mathrm{id}$ for some $\alpha\in \mathbb{Z}[i]$ with $|\alpha|>1$. The proof goes exactly the same as in Lemma~\ref{lem: essential dimension integral matrix and factorization}, noticing that $\mathbb{Z}[i]$ is a principal ideal domain. 
\end{remark}

As a corollary of Lemma~\ref{lem: essential dimension integral matrix and factorization} or Proposition \ref{prop: essential dimension be one criterion}, we construct in Example \ref{ex: counterexample for stronger version of KZ question} below a polarized self-isogeny $f$ of an abelian surface $E\times E$ such that $\mathrm{ed}(f^s)<2$ for all positive integers $s$. {\it This gives a counterexample to Theorem~\ref{thm: ed main} if we do not require that all subtorus are $f$-periodic.}

\begin{example}\label{ex: counterexample for stronger version of KZ question} 
Now we construct polarized endomorphisms $f$ of $A=E\times E$ satisfying the property that 
$$\deg(f)\geq 2, \hskip 1pc \mathrm{ed}(f^s)<2$$
for all positive integers $s$ (this includes non-simple abelian surfaces with Picard number three or four). To do so, by Lemma~\ref{lem: essential dimension integral matrix and factorization}(1) (a consequence of Proposition \ref{prop: essential dimension be one criterion}), it suffices to find
\begin{enumerate}
    \item a $2\times 2$ diagonalizable integer matrix $M$ such that $|\mathrm{det}(M)|>1$ and the eigenvalues of $M$ have the same modulus, and
    \item an integer vector $V=(m_0,n_0)$ such that for all positive integers $s$, $V_s:=V\cdot M^s=(m_s,n_s)$ satisfies $\mathrm{gcd}(m_s,n_s)=1$.
\end{enumerate}
Indeed, we consider the self-isogeny $f=f_M$ of $A:=E\times E$ represented by multiplication by the matrix $M$ on the universal covering space of $A$. Then $\deg(f)=|\mathrm{det}(M)^2| >1$. By Lemma~\ref{lem: eigenvalues of polarized map}, Condition $(1)$ above implies that $f$ is numerically $q$-polarized. Hence $f^2$ is polarized. By Lemma~\ref{lem: essential dimension integral matrix and factorization}(1), Condition $(2)$ above guarantees that $\mathrm{ed}(f^s)\leq 1$ for all positive integers $s$. It follows that $f^2$ satisfies the desired property.

There are many matrices and vectors satisfying Conditions $(1)$ and $(2)$ above. For example, let $$M=\begin{pmatrix} 0 & 1 \\ -t & 1 \end{pmatrix}\ \ \ \text{\rm and}\ \ \ V=(1, 0)$$ for each integer $t>1$. We first check Condition $(1)$. Since the characteristic polynomial of $M$ is $x^2-x+t$, the two eigenvalues are conjugate to each other and hence have the same modulus. For Condition $(2)$, we set $V_s:=V\cdot M^s=(m_s,n_s)$ as before, and we inductively prove that $m_s+n_s\equiv 1\mod t$ and $\mathrm{gcd}(m_s,n_s)=1$. This holds for $s=0$.  Since $$m_{s+1}+n_{s+1}\equiv m_s+n_s\mod t\ \ \ and \ \ \ \mathrm{gcd}(m_{s+1},n_{s+1})=\mathrm{gcd}(m_s+n_s,-tn_s)=\mathrm{gcd}(m_s,n_s),$$ we obtain Condition $(2)$ by induction.
\end{example}

\begin{remark}\label{rem: rank is not ed}
    Example~\ref{ex: counterexample for stronger version of KZ question} also gives an example of a polarized isogeny such that $\rank \Ker(f)\geq \dim(A)$ (see Lemma~\ref{lem: real characteristic polynomial}) and $\mathrm{ed}(f)<\dim(A)$.
\end{remark}

The following lemma is motivated by a communication with S. Meng, and we state a generalized version here:

\begin{lemma}\label{lem: integer matrix coprime index}
    Let $R$ be a commutative ring and let $A\in M_{n\times n}(R)$ be an $n\times n$ matrix. Suppose that the ideal generated by all entries of $A^n$ is $R$. Then the ideal generated by all entries of $A^{k}$ is $R$ for all $k\in \mathbb{Z}_{>0}$.
\end{lemma}

\begin{proof}
    We may assume that $R\neq 0$. Otherwise the statement is trivial.
    
    Let $I_k$ be the ideal generated by all entries of $A^{k}$ for each $k$. It follows that $I_k\subset I_{k'}$ for each $k'\leq k$. Hence $I_n=R$ implies that $I_k=R$ for all $k\leq n$. Now we prove that $I_k=R$ for $k>n$. Suppose to the contrary that $I_k\subsetneq R$ for some $k>n$. Then we may take a maximal ideal $\mathfrak{m}_k\subsetneq R$ containing $I_k$, with the residue field $\kappa=R/\mathfrak{m}_k$. We let $A_\kappa$ be the reduction of $A$ modulo $\mathfrak{m}_k$. Then $A_\kappa\in M_{n\times n}(\kappa)$. Moreover, $A_\kappa^k=0$ over $\kappa$. Note that $\Ker (A_\kappa^k)=\Ker (A_\kappa^n)$ when $k\geq n$, hence $A_\kappa^n=0$. This implies that $I_n\subset \mathfrak{m}_k$, a contradiction.
\end{proof}

\begin{remark}\label{rem: simple lemma to check kernel is cyclic}
    There are several interesting applications of Lemma~\ref{lem: integer matrix coprime index}.
    \begin{enumerate}
        \item Consider the elliptic curve $E:=\C/\Lambda$ where $\Lambda=\Z\oplus \Z\tau$ for some quadratic integer $\tau\in \C\setminus \R$. We consider the isogeny $f_c$ of $E$ defined by $z\to c\cdot z$ for some $c\in \mathrm{End}(E)$. Then there exists an integer matrix $M_c$ such that $c\cdot (1,\tau)=(1,\tau)\cdot M_c$. It is easy to see that $\deg(f_c^s)=|\det(M_c)|^s$ for all positive integers $s$, and $\Ker(f^s)$ is a cyclic group if and only if the greatest common divisor of the entries of $M_c^s$ is $1$. By applying Lemma~\ref{lem: integer matrix coprime index} to $M_c$ with $R=\Z$, we conclude that $\rank \Ker(f^s)=1$ for all positive integers $s$ if and only if $\rank \Ker(f^2)=1$.
        \item For $p+qi\in \Z[i]$, set $p_s+q_si=(p+qi)^s$. Then $\gcd(p_s,q_s)=1$ for all positive integers $s$ if and only if $\gcd(p_2,q_2)=1$.
    \end{enumerate}
\end{remark}

As a corollary of Lemma~\ref{lem: integer matrix coprime index}, we have:

\begin{lemma}\label{lem: essential dimension index}
    Let $E$ be an elliptic curve and $A=E\times E$. Let $f$ be a self-isogeny of $A$ defined by $$f: A\to A,\ \ \ f(x,y)=(x,y)\cdot M,$$ for some integer matrix $M$ and some coordinates $(x,y)$ on the universal cover of $A$. Suppose that $\mathrm{ed}(f^s)=2$ for some integer $s>0$. Then $\mathrm{ed}(f^t)=2$ for all $t \ge 2$.
\end{lemma}

\begin{proof}
    This follows directly from Lemma~\ref{lem: essential dimension integral matrix and factorization}(2), Lemma~\ref{lem: integer matrix coprime index} (with $R:=\Z$) and Lemma \ref{lem: facts of essential dimension}.
\end{proof}

\begin{remark}
 Suppose that $f$ is represented by the matrix $$\begin{pmatrix} 1 & 1 \\ 1 & -1 \end{pmatrix}$$ in Lemma~\ref{lem: essential dimension index}. Then $\mathrm{ed}(f)=1$ by Lemma~\ref{lem: essential dimension integral matrix and factorization}(2), while $f^2=2I$, hence $\mathrm{ed}(f^2)=2$. This shows that the power index two is optimal in Lemma~\ref{lem: essential dimension index}.
\end{remark}

The following lemma provides many examples of incompressible self-isogenies that do not factor through a scalar map.

\begin{lemma}\label{lem: essential dimension isogenous product integer matrix}
    Let $A=E_1\times E_2$, where $E_1$ and $E_2$ are two non-isogenous elliptic curves. Let $f$ be a self-isogeny of $A$ and hence $f$ splits as $f_1\times f_2$, where $f_i$ is a self-isogeny of $E_i$. Suppose further that $f_i\in \mathrm{End}(E_i)\cap \mathbb{R}$ for each $i\in \{1,2\}$. Then $\mathrm{ed}(f)=2$ if and only if $\mathrm{min}\{\deg (f_1),\deg(f_2)\}\geq 2$.
\end{lemma}

\begin{proof}
Suppose that $\mathrm{ed}(f)=2$. Then it follows directly that $\mathrm{min}\{\deg (f_1),\deg(f_2)\}\geq 2$. Now suppose that $\mathrm{min}\{\deg (f_1),\deg(f_2)\}\geq 2$. We prove that $\mathrm{ed}(f)=2$. By our assumption, $f$ splits as $m_1\times m_2$, where $(f_i =$) $m_i$ denotes the multiplication-by-$m_i$ map on $E_i$ for each $i$. Here the $m_j$ are integers. Indeed, it is clear if $E_j$ is non-CM type. While if $E_j$ is CM type, its period is complex and the assumption on 
$f_j$ (preserving the lattice) implies that $f_j = m_j$ is an integer.

If $\mathrm{gcd}(m_1,m_2)>1$, then $f$ factors through a multiplication-by-$m$ map for some $m>1$. Hence $\mathrm{ed}(f)=2$ by \cite[Theorem~1]{KZ} and Lemma~\ref{lem: facts of essential dimension}(3). From now on, we assume that $\mathrm{gcd}(m_1,m_2)=1$.

By \cite[Theorem~2]{KZ}, for some subtorus $B$ of $A$ and any prime number $p$, we have 
\begin{equation}\label{eq: Kollar Zhuang inequality}
    \mathrm{ed}(f) \geq \dim(A) - \dim(B) + \frac{p-1}{p} \rank_p(\Ker (f) \cap B)
\end{equation} 
where $\rank_p G$ denotes the $p$-rank of an abelian group $G$. By our assumption on $A$, there are only four possibilities for $B$, namely $B\in \{\{0\}, A,E_1\times\{0\}, \{0\}\times E_2 \}$.

If $B=\{0\}$, then by \eqref{eq: Kollar Zhuang inequality}, $\mathrm{ed}(f)\geq \dim(A)$, hence $\mathrm{ed}(f)=2$ in this case.

If $B=E_1\times\{0\}$ or $\{0\}\times E_2$, then the restriction $f|_B=m_i$ on $E_i$ for $i\in \{1,2\}$. In particular, $\Ker (f) \cap B\cong (\mathbb{Z}/m_i\mathbb{Z})^2$. We may choose $p$ so that $p\mid m_i$. Then \eqref{eq: Kollar Zhuang inequality} implies that $\mathrm{ed}(f)\geq 2-1+\frac{p-1}{p}\times 2$. Hence $\mathrm{ed}(f)=2$.

If $B=A$, then $\Ker (f) \cap B\cong (\mathbb{Z}/m_1\mathbb{Z})^2\times (\mathbb{Z}/m_2\mathbb{Z})^2$. 
Since $\mathrm{gcd}(m_1,m_2)=1$ and $m_j \ge 2$, there exists a prime $p>2$ such that $p\mid m_1m_2$. Then $\rank_p(\Ker (f) \cap B)= 2$, and $\mathrm{ed}(f)\geq 2\frac{p-1}{p}>1$. Hence $\mathrm{ed}(f)=2$.
\end{proof}

Let $f: X \to X$ be a surjective endomorphism of a projective variety $X$. Recall that $f$ is \emph{int-amplified} if $f^*L-L=H$ for some ample Cartier divisors $L$ and $H$. The following example shows that Theorem~\ref{thm: ed main} does not hold if we only require that $f$ is int-amplified.

\begin{example}\label{ex: cyclic after iteration example int-amplified}
Let $E_i=\C/\Lambda_i$, where $\Lambda_1=Z[i]$ and $\Lambda_2=Z[\sqrt{-2}]$. We consider the self-isogenies $h_1: z_1\to (1+2i)z_1$ and $h_2: z_2\to (1+\sqrt{-2})z_2$ defined on $E_1$ and $E_2$ respectively. Note that $\deg(h_1)=5$ and $\deg(h_2)=3$. Note also that $h_1^s$ and $h_2^s$ has cyclic kernels $\Z/5^s\Z$ and $\Z/3^s\Z$ respectively (see Remark~\ref{rem: simple lemma to check kernel is cyclic}(1)). Now, let $h=h_1\times h_2$. Then $h$ is an int-amplified self-isogeny of $E_1\times E_2$ such that every one-dimensional subtorus, which are just $E_1 \times \{0\}$ and $\{0\} \times E_2$, is $h$-stable. However, $\Ker(h^s)=\Z/15^{s}\Z$ is cyclic for any $s \ge 1$, hence $\mathrm{ed}(h^s)=1$ for all positive integers $s$ by Lemma~\ref{lem: facts of essential dimension}(4).
\end{example}

We end this section with a counterexample for Question~\ref{ques: polarized incompressible} when the base variety is rationally connected.

\begin{example}\label{ex: counterexample rationally connected}
    Consider the endomorphism $f$ on $\mathbf{P}^1\times \mathbf{P}^1$ given by $(x,y) \mapsto (y,x^4)$. Then $f$ is a $\Z/4\Z$ cover. Moreover, $f^2: (x,y) \mapsto (x^4,y^4)$ is $4$-polarized, and $f^2$ is a $\Z/4\Z\times \Z/4\Z$ cover. By \cite[Proof of Theorem 2.7, Note ]{Zh10}, $f$ is $2$-polarized. By Remark~\ref{rem: kernel and incompressible}(1), $\mathrm{ed}(f)\leq 1$. Note that if $\mathrm{ed}(f^2)<2$, then $f^2$ is a birational base change of a finite morphism between $\mathbf{P}^1$. In particular, $\Z/4\Z\times \Z/4\Z$ acts faithfully on $\mathbf{P}^1$. This leads to a contradiction as $\mathrm{PGL}(2)$ does not have a subgroup that is isomorphic to $\Z/4\Z\times \Z/4\Z$. Hence $\mathrm{ed}(f^2)=2$.
\end{example}

\subsection{Proof of {Theorem~\ref{thm: rank main uniform}}}
The goal of this section is to prove Theorem~\ref{thm: rank main uniform}. It suffices to work over $\C$ by the Lefschetz principle.

We start with the following lemma which plays a key role in this section.

\begin{lemma}\label{lem: real characteristic polynomial}
    Let $A$ be an abelian variety of dimension $n$, and let $f: A\to A$ be a self-isogeny defined by $$f: (z_1,\cdots,z_n)^t \mapsto f(z_1,\cdots,z_n):=M\cdot (z_1, \cdots, z_n)^t$$
    for some $M \in M_{n\times n}(\C)$, where $(z_1, \cdots,z_n)$ are the coordinates on the universal cover $\C^n$ of $A$.
    Suppose that one of the following conditions holds:
    \begin{enumerate}
        \item The characteristic polynomial $p_M(x)$ of $M$ has coefficients in $\R$.
        \item $M\in M_{n\times n}(\R)$.
    \end{enumerate}
    Then $p_M(x)$ is in $\Z[x]$ (hence $\det(M)\in \Z$),  and $\mathrm{exp}(\Ker(f))\mid \det(M)$. Moreover, for any prime $p\mid \det(M)$, $\rank_p\Ker(f)\geq 2.$
\end{lemma}

\begin{proof}
    Obviously, Condition (2) implies Condition (1). Thus, we may assume Condition (1) holds.

    Write $A=\mathbb{C}^n/\Lambda$, where $\Lambda$ is a lattice generated by $2n$ complex (column) vectors $v_1,\cdots,v_{2n}$. Since $f$ is a self-isogeny of $A$, we have $f(\Lambda)\subset \Lambda$, and hence 
    \begin{align}\label{eq: integer relation representation}
        M\cdot (v_1,\cdots,v_{2n})= (v_1,\cdots,v_{2n})\cdot L
    \end{align}
    for some $2n\times 2n$ integer matrix $L$. Note that the matrix of $f$ with respect to the standard $\mathbb{R}$-basis of $\mathbb{C}^n$ given by $$\{e_k,ie_k\mid e_k \text{ is the standard Euclidean basis vector whose $k$-th coordinate is 1}\}_{1\leq k\leq n}$$ is $\mathrm{diag}(M,M)$. Let $p_M(x)=\det(xI_n-M)\in \R[x]$ be the characteristic polynomial of $M$. Then \eqref{eq: integer relation representation} implies that 
    \begin{align}\label{eq: equation on characteristic polynomial}
        p_M(x)^2=\det(xI_{2n}-\diag(M,M))=\det(xI_{2n}-L)\in \Z[x].
    \end{align}
    By descending induction on the degree of each monomial term of $p_M(x)$, we conclude from \eqref{eq: equation on characteristic polynomial} that $p_M(x)\in \Q[x]$. Let $m$ be a rational number such that $mp_M(x)$ is a primitive element in $\Z[x]$. By Gauss's lemma, $(mp_M(x))^2$ is primitive. However, $p_M(x)^2=\det(xI_{2n}-L)$ is also primitive. This implies $m=1$, and therefore $p_M(x)\in \Z[x]$. In particular, $\det(M)$ is an integer. Note that $|\Ker(f)|=\deg(f)=|\det(M)|^2$.
    
    By the Cayley–Hamilton theorem, the map $p_M(f)$ is the zero map as well as a self-isogeny on $A$. This implies that for any $x\in \Ker(f)$, $$0=p_M(f)(x)=(-1)^n \det(M)\cdot x.$$ It follows that $\mathrm{exp}(\Ker(f))\mid \det(M)$. In particular, we have $\rank \Ker(f)\geq 2$. For a prime $p\mid \det(M)$, we may write $\det(M)=up^v$ for some $u,v\in \Z$ such that $\gcd(u,p)=1$. Let $\Ker_p(f)$ be the Sylow $p$-subgroup of $\Ker(f)$. Then $\exp(\Ker_p(f))\leq p^v< p^{2v}=|\Ker_p(f)|$. This implies $\rank_p \Ker(f)\geq 2$.
\end{proof}

Recall that Albert classified the endomorphism algebras $\mathrm{End}(A)\otimes \Q$ for simple abelian varieties $A$ into four types (\cite{Alb34,Alb35}, see also \cite[Proposition 5.5.7]{LB92} for a summary). Later, Shimura described the moduli of abelian varieties whose endomorphism algebra corresponds to each of these four types (\cite{Shi63}, see \cite[Chapter 9]{LB92} for a summary). Before stating and proving Proposition~\ref{prop: polarized rank simple surface} (Theorem~\ref{thm: rank main uniform}(2)), we recall some facts about \emph{quaternion algebras}.

\begin{definition}\label{defn: quaternion}
    Fix a field $k$ with $\mathrm{char}~k\neq 2$. A \emph{quaternion algebra} $F=F_k(a,b)$ over $k$ is defined as $F_k(a,b):=k\langle i,j\rangle/(i^2-a, j^2-b, ij = -ji)$ for $a,b\in k\setminus\{0\}$. We denote $ij$ by $u$. Then elements in $F_k(a,b)$ can be written as $$F_k(a,b)=\{x_1+x_2i+x_3j+x_4u\mid x_1,x_2,x_3,x_4\in k\}.$$ We have a natural representation of quaternion algebras
    $$\rho: F_k(a,b)\to M_{2\times 2}(\bar k),\ \ \ \rho(x)=\begin{pmatrix} 
    x_1+x_2\sqrt{a} & b(x_3-x_4\sqrt{a})\\ 
    x_3+x_4\sqrt{a} & x_1-x_2\sqrt{a}
    \end{pmatrix},$$ where $\bar k$ is the algebraic closure of $k$ and $x=x_1+x_2i+x_3j+x_4u$. Moreover, we have an anti-involution $x^*:= x_1-x_2i-x_3j-x_4u$ on $F_k(a,b)$. We define the \emph{trace} and the \emph{norm} of a quaternion $x$ as $T(x) = x + x^*\in k$ and $N(x) = xx^*\in k$, respectively.
\end{definition}

\begin{proposition}\label{prop: polarized rank simple surface}
    Let $A$ be a simple abelian surface, and let $f$ be a $q$-polarized self-isogeny of $A$. Then $\exp(\Ker(f))\mid q$. Moreover, for any prime $p\mid q$, we have $\rank_p\Ker(f)\geq 2.$
\end{proposition}

\begin{proof}
        Let $F:=\mathrm{End}(A)\otimes \mathbb{Q}$. Then the Rosati involution provides a positive anti-involution $*$ on $F$. Let $K$ be the center of $F$, and let $$d^2:=[F:K]\ \ \ \text{and}\ \ \ e:=[K:\Q].$$
        Albert classified all such pairs $(F,*)$ for simple abelian varieties into four types (see \cite{Alb34}, \cite{Alb35} and \cite[Proposition~5.5.7]{LB92}). Later, Murty described the Picard number for simple abelian varieties in each of the four cases (\cite[Lemma~3.3]{Mur84}). In particular, for simple abelian surfaces, we summarize the results as follows:
        \begin{itemize}
            \item \emph{Type I}: $F=K$ is a totally real number field.
            \item \emph{Type II}: $K=\Q$ and $F$ is a totally indefinite quaternion algebra over $\Q$. More specifically, $F=F_\Q(a,b)$ for some $a,b\in \Q_{>0}$.
            \item \emph{Type III}: $K\subset \R$, and $F$ is a totally definite quaternion algebra over $K$. More specifically, $F=F_K(a,b)$ for some $a,b\in K_{<0}$.
            \item \emph{Type IV}: $F=K$ is totally complex.
        \end{itemize}

        Now we consider the self-isogeny $f$ of $A$ with $\deg(f)\geq 2$ defined by $$f: (z_1,z_2)^t\mapsto f(z_1,z_2):=M \cdot (z_1, z_2)^t$$
        for some $M  \in M_{2\times 2}(\C)$, where $(z_1, z_2)$ are the coordinates on the universal cover $\C^2$ of $A$. Let $p_M(x)$ be the characteristic polynomial of $M$. Note that $|\Ker(f)|=\deg(f)=|\det(M)|^2$.
        
        We claim that $p_M(x)\in \R[x]$ if we are in \emph{Type I, II}, or \emph{III}. If we are in \emph{Type I}, then $M$ is a real matrix, and hence $p_M(x)\in \R[x]$ in this case. If we are in \emph{Type II} or \emph{Type III}, then $f$ is represented by a quaternion $q_f$ over a real field $K\subset \R$, and $M=\rho(q_f)^t$. In this case, $$p_M(x)=x^2-T(q_f)x+N(q_f)$$ with $T(q_f), N(q_f)\in K\subset \R$. It follows that $p_M(x)\in \R[x]$. This proves the claim. By Lemma~\ref{lem: real characteristic polynomial}, the assertion holds if we are in \emph{Type I, II}, or \emph{III}.

         If we are in \emph{Type IV}, then $f$ is a multiplication-by-$\lambda$ map for some complex number $\lambda$. By Lemma~\ref{lem: complex structure of endomorphism ring of abelian variety}, $\exp(\Ker(f))\mid q$. The `Moreover' part follows by the same argument as in Lemma~\ref{lem: real characteristic polynomial}.
\end{proof}

\begin{proposition}\label{prop: picard num 3 product case}
    Let $A$ be a non-simple abelian surface of Picard number three. Let $f$ be a self-isogeny of $A$ with $\deg f \geq 2$. Then $\rank \Ker(f)\geq 2$.
\end{proposition}

\begin{proof}
By \cite{SM74} and \cite[Lemma~3.3]{Mur84} (or \cite[Corollary~2.6]{HL19}), $A$ is isogenous to $E\times E$, where $E$ is an elliptic curve without complex multiplication. In particular, $\mathrm{End}(E)\otimes \mathbb{Q}=\mathbb{Q}$. Since $A$ and $E\times E$ are isogenous, $$\mathrm{End}(A)\otimes \mathbb{Q}=\mathrm{End}(E\times E)\otimes \mathbb{Q}=M_{2\times 2}(\mathbb{Q}).$$ It follows that $f$ is represented by a $2\times 2$ matrix with rational coefficients. The assertion then follows from Lemma~\ref{lem: real characteristic polynomial}. 
\end{proof}

Now we prove Theorem~\ref{thm: rank main uniform}(1) when the Picard number $\rho(A)=4$. The main ingredient in the proof is Shioda and Mitani's structure theorem of abelian surfaces with Picard number four.

\begin{proposition}\label{prop: picard num 4}
    Let $A$ be an abelian surface such that $\rho(A)=4$. Then there is a positive integer $s$ depending only on $A$, such that for any polarized self-isogeny $f$ of $A$, $\rank\Ker(f^s)\geq 2$. 
\end{proposition}

\begin{proof}
Since $\rho(A)=4$, by \cite[Theorem~4.1(iii) and the second paragraph on Page~272]{SM74}, there exist positive integers $a,b,c$ such that $\Delta:=b^2-4ac<0$, and $$A\cong \mathbb{C}/\Lambda_{\tau_1}\times\mathbb{C}/\Lambda_{\tau_2},$$ where $\Lambda_{\tau_i}=\mathbb{Z}\oplus \mathbb{Z}\tau_i$ for $i\in \{1,2\}$, $\tau_1=\frac{-b+\sqrt{\Delta}}{2a}$ and $\tau_2=\frac{b+\sqrt{\Delta}}{2}$. Note that $\tau_2=a\tau_1+b$. Hence $\Lambda_{\tau_2}$ and $\Lambda_{a\tau_1}$ are exactly the same lattice. Replacing $\tau_2$ with $a\tau_1$ and abbreviating $\tau_1$ as $\tau$, we may assume that $A\cong \mathbb{C}/\Lambda_{\tau}\times\mathbb{C}/\Lambda_{a\tau}$, where
\begin{align}\label{eq: properties of tau}
    a\tau^2+b\tau+c=0.
\end{align}

Now, we let $g$ be any polarized surjective self-isogeny of $A$ such that $\Ker (g^2)$ (and hence $\Ker(g)$) is cyclic. We may write $g$ as a matrix multiplication map $$g: (z_1,z_2)^t \mapsto g(z_1,z_2):=M\cdot (z_1,z_2)^t,\ \ M\in M_2(\mathbb{C}).$$ 
$\Lambda:=\Lambda_{\tau} \times \Lambda_{a \tau}$ has a $\Z$-basis $\{(1, 0)^t, (\tau, 0)^t, (0, 1)^t, (0, a \tau)^t\}$.
Since $g$ is surjective, $\mathrm{det}(M)\neq 0$. 
\begin{claim}\label{claim: ration of eigenvalues in finite set}
    Let $\lambda_1,\lambda_2$ be the  complex eigenvalues of $M$. Then $\frac{\lambda_1}{\lambda_2}$ belongs to a finite set $\Gamma(A)$ depending only on $A$.
\end{claim}

We first prove Proposition~\ref{prop: picard num 4} assuming the above claim. We set $I:=[1,|\Gamma(A)|+1]\cap \mathbb{Z}$, where $|\Gamma(A)|$ denotes the cardinality. Suppose to the contrary that $\Ker (f^s)$ is cyclic for all $s\in I$. We let $\lambda_1,\lambda_2$ be the two eigenvalues of the complex matrix representing $f$. Then $\lambda_1^s, \lambda_2^s$ are the eigenvalues of the complex matrix representing $f^s$ for each $s$. By Claim~\ref{claim: ration of eigenvalues in finite set}, $$\{\frac{\lambda_1^s}{\lambda_2^s}\mid s\in I\}\subset \Gamma(A).$$ Hence, by the pigeonhole principle, there exist two distinct positive integers $s_1,s_2\in I$ such that $\frac{\lambda_1^{s_1}}{\lambda_2^{s_1}}=\frac{\lambda_1^{s_2}}{\lambda_2^{s_2}}$. It follows that $\lambda_1^{|s_1-s_2|}=\lambda_2^{|s_1-s_2|}=\lambda$ for some $\lambda\in \mathbb{C}$. For simplicity, we write $s=|s_1-s_2|\in I$. It follows that $f^s$ is represented by $\lambda I$, a complex scalar multiplication map. In particular, $f:=h_1\times h_2$, where $h_1$ and $h_2$ are the multiplication-by-$\lambda$ maps on $\mathbb{C}/\Lambda_{\tau}$ and $\mathbb{C}/\Lambda_{a\tau}$, respectively. Hence $\Ker (f^s)=\Ker (h_1)\times \Ker (h_2)$, where $|\Ker (h_1)|=|\Ker (h_2)|=|\lambda|^2$. This implies $\mathrm{rank} \Ker(f^s)\geq 2$. This contradicts our assumption that $\rank\Ker(f^s)=1$. We are done.

\begin{proof}[Proof of Claim~\ref{claim: ration of eigenvalues in finite set}]
    Notice that $g$ induces a lattice self-map on $\Lambda$. It follows that $g(1,0)^t\in\Lambda$ and $g(0,1)^t\in \Lambda$, where $\Lambda=\Lambda_{\tau}\times \Lambda_{a\tau}$. Hence $$M=\begin{pmatrix} \alpha_{11} + \beta_{11}\tau & \alpha_{12} + \beta_{12}\tau \\ \alpha_{21} + \beta_{21}a\tau & \alpha_{22} + \beta_{22}a\tau \end{pmatrix}$$ for some integers $\alpha_{ij},\beta_{ij}$ with $i,j\in \{1,2\}$. Since $g(\tau,0)^t\in \Lambda$ and $g(0,a\tau)^t\in \Lambda$, 
\begin{align}\label{eq: condition on coeff of g}
     (\alpha_{11} + \beta_{11}\tau)\cdot\tau, \ (\alpha_{12} + \beta_{12}\tau)\cdot a\tau \in \Lambda_{\tau}\ \ and\ \ (\alpha_{21} + \beta_{21}a\tau)\cdot \tau, \ (\alpha_{22} + \beta_{22}a\tau)\cdot a\tau \in \Lambda_{a\tau}.
\end{align}
Since the multiplication by $M$ map has degree $|\mathrm{det}(M)|^2$ and $\Ker (M)$ is a cyclic group, there exists an element in $\Ker (M)$ whose order is equal to $|\mathrm{det}(M)|^2$. Notice that 
$$\Ker (M)=\{M^{-1}\cdot (m_1+n_1\tau,m_2+n_2a\tau)^t\mid m_1,m_2,n_1,n_2\in \mathbb{Z}\},$$
and 
$$M^{-1}=\mathrm{det}(M)^{-1}M^*=\frac{\overline{\det(M)}}{|\det(M)|^2}\begin{pmatrix} \alpha_{22} + \beta_{22}a\tau & -(\alpha_{12} + \beta_{12}\tau) \\ -(\alpha_{21} + \beta_{21}a\tau) & \alpha_{11} + \beta_{11}\tau \end{pmatrix},$$
where $M^*$ denotes the adjoint matrix of $M$. By \eqref{eq: condition on coeff of g}, we know that $M^*(\tau,0)^t \in \Lambda$ and $M^*(0,a\tau)^t\in \Lambda$. Hence the multiplication by $M^*$ defines a lattice self-map of $\Lambda$. It follows that $$\Ker (M)\subset \{\frac{\overline{\det(M)}}{|\det(M)|^2}(m_1+n_1\tau,m_2+n_2a\tau)^t\mid m_1,m_2,n_1,n_2\in \mathbb{Z}\}.$$ 
Note that by \eqref{eq: properties of tau}, $$\det(M)=(\alpha_{11}+\beta_{11}\tau)(\alpha_{22}+\beta_{22}a\tau)-(\alpha_{12}+\beta_{12}\tau)(\alpha_{21}+\beta_{21}a\tau)\in \Lambda_{\tau}.$$ We may write $\det(M)=p+q\tau$ for some integer $p$ and $q$, and we set $l=\mathrm{gcd}(p,q,|\det(M)|^2)$. By \eqref{eq: properties of tau}, $\tau\bar\tau=\frac{c}{a}$ and $\tau+\bar\tau=-\frac{b}{a}$. This implies that for each $(z_1,z_2)\in \Lambda = \Lambda_{\tau}\times \Lambda_{a\tau}$, by actual multiplication, we get
$$\frac{1}{|\mathrm{det}(M)|^2}\begin{pmatrix} \overline{\det(M)}& 0 \\ 0 & \overline{\det(M)} \end{pmatrix}\cdot (z_1,z_2)^t \in \frac{l}{a|\det(M)|^2}\Lambda_{\tau}\times \frac{l}{a|\det(M)|^2}\Lambda_{a\tau}.$$ Hence $l=\mathrm{gcd}(p,q,|\det(M)|^2)\mid a$ in order that there exists an order $|\det(M)|^2$ element in $\Ker (M)$. Notice also that by \eqref{eq: properties of tau}, $$|\det(M)|^2=(p+q\tau)(p+q\bar\tau)=p^2+\frac{c}{a}q^2-\frac{b}{a}pq = p^2 + \frac{cq^2 - bpq}{a}.$$ It follows that
\begin{align}\label{eq: gcd of det M}
    \mathrm{gcd}(p,q)\mid a \, \mathrm{gcd}(p,q,\frac{cq^2-bpq}{a})=a \, \mathrm{gcd}(p,q,|\det(M)|^2)\mid a^2.
\end{align}
We let $\lambda_1, \lambda_2$ be the two complex eigenvalues of $M$. By Lemma~\ref{lem: eigenvalues of polarized map}, $g$ being polarized implies that $M$ is diagonalizable and $|\lambda_1|=|\lambda_2|$. In particular, $\lambda_1+\lambda_2$ is perpendicular to $\lambda_1-\lambda_2$ as real vectors on $\mathbb{C}$. Then either $\lambda_1 = \pm \lambda_2$, in which case $g^2$ 
is represented by a diagonal matrix and hence $\rank \Ker(g^2) \ge 2$ (contradicting the assumption), or $\lambda_1 \ne \pm \lambda_2$ and $i\mu(\lambda_1+\lambda_2)=\lambda_1-\lambda_2$ for some real number $\mu$. It follows that $$\frac{\mathrm{det}(M)}{\mathrm{tr}(M)^2}=\frac{\lambda_1\lambda_2}{(\lambda_1+\lambda_2)^2}=\frac{(\lambda_1+\lambda_2)^2-(\lambda_1-\lambda_2)^2}{4(\lambda_1+\lambda_2)^2}=\frac{1+\mu^2}{4}\geq \frac{1}{4}.$$
By \eqref{eq: properties of tau} and \eqref{eq: condition on coeff of g}, we have 
\begin{align*}
    \mathrm{tr}(M)^2&=(\alpha_{11}+\beta_{11}\tau+\alpha_{22}+a\beta_{22}\tau)^2 \\
    &\equiv (\alpha_{11}+\beta_{11}\tau)\cdot (\beta_{11}+a\beta_{22})\tau+(\alpha_{22}+a\beta_{22}\tau)\cdot (\beta_{11}+a\beta_{22})\tau\equiv 0           \mod \Lambda_{\tau}.
\end{align*}
Hence $\mathrm{tr}(M)^2\in \Lambda_{\tau}$, and we may write $\mathrm{tr}(M)^2=x+y\tau$ for some integers $x,y$. Recall that $\det(M)=(p'+q'\tau)\mathrm{gcd}(p,q)$ for some coprime integers $p',q'$ and that $\mathrm{gcd}(p,q)\mid a^2$. We have $$\R \, \ni \, \frac{\mathrm{det}(M)}{\mathrm{tr}(M)^2}=\frac{\lambda_1\lambda_2}{(\lambda_1+\lambda_2)^2}=\mathrm{gcd}(p,q)\frac{p'+q'\tau}{x+y\tau}\in \{\frac{a^2}{n} \mid n\in \mathbb{Z}\}\cap[\frac{1}{4},\infty) =: \Gamma(A),$$ where the last inclusion relation follows from that $p'+q'\tau$ is a primitive $\mathbb{Z}$-vector in $\Lambda_{\tau}$: indeed, write  
$(p' + q' \tau)/(x + y \tau) = r \in \R$, we get $r = p'/x = q'/y$. 
Hence $\frac{\lambda_1\lambda_2}{(\lambda_1+\lambda_2)^2} = \frac{\nu}{(1 + \nu)^2}$ with $\nu = \lambda_1/\lambda_2$ belongs to the above displayed finite set $\Gamma(A)$ with cardinality at most $4a^2$, and consequently $\nu = \frac{\lambda_1}{\lambda_2}$ belongs to a finite set with cardinality at most $8a^2$.
\end{proof}
\end{proof}

Now we are ready to prove Theorem~\ref{thm: rank main uniform}.

\begin{proof}[Proof of Theorem~\ref{thm: rank main uniform}]
    (2) follows from Proposition~\ref{prop: polarized rank simple surface}.
    
    We prove (1) according to $\rho(A)$. Notice that $\rho(A)\leq 4$. The case $\rho(A)\in \{1,2\}$ follows from Proposition~\ref{prop: rank main 1 uniform}. The case $\rho(A)=3$ follows from Propositions~\ref{prop: polarized rank simple surface} and \ref{prop: picard num 3 product case}. The case $\rho(A)=4$ follows from Proposition~\ref{prop: picard num 4}.
\end{proof}

\subsection{Uniform incompressible index in dimension two: Proof of Theorem~{\ref{thm: simple abelian surface uniform index}}}

\begin{proposition}\label{prop: uniform surface incompressible index stable}
    Let $A$ be an abelian surface and let $f$ be a $q$-polarized endomorphism of $A$. Suppose that every one-dimensional subtorus of $A$ is numerically $f$-stable. Then $f$ is either incompressible or $2$-polarized.
\end{proposition}

\begin{proof}
As in Theorem~\ref{thm: main with all subtori f numerically periodic}, by replacing $f$ with its translation, we may assume that $f$ is a self-isogeny and every one-dimensional subtorus of $A$ is numerically $f$-stable. We discuss two cases.

\noindent {\bf Case 1}. $A$ is simple. By Proposition~\ref{prop: polarized rank simple surface}, $\exp(\Ker(f))\mid q$. This implies that $\exp(\Ker(f^s))\mid q^s$ for all positive integers $s$. By Lemma~\ref{lem: main lemma with all subtori f stable} and Remark~\ref{rem: dim two case choose the constant J} (where we can take $J=2$ and $C=1$), $f^s$ is incompressible for $s>\log_q (2)$.

\noindent {\bf Case 2}. $A$ is non-simple. In this case, $A$ is isogenous to the product of two elliptic curves. By Lemma~\ref{lem: basic property of polarized end of product abelian}, $\exp(\Ker(f))\mid q$. Then the same argument as in {\bf Case 1} implies that $f^s$ is incompressible for $s>\log_q (2)$. 

If $q>2$, then $\log_q (2)<1$, and hence $f$ is incompressible in both cases. We are done.
\end{proof}

As a direct corollary of Proposition~\ref{prop: uniform surface incompressible index stable}, we have the following, which includes Theorem~\ref{thm: simple abelian surface uniform index}.

\begin{corollary}\label{cor: corollary of stable incompressible in dimension two}
    Let $A$ be an abelian surface and let $f$ be a polarized endomorphism $f$ of $A$. Then the following statements hold.
    \begin{enumerate}
        \item If $A$ is simple, then either $f$ is incompressible or $f$ is $2$-polarized.
        \item If the Picard number $\rho(A)\leq 2$, then $f^2$ is incompressible.
    \end{enumerate}
\end{corollary}
\begin{proof}
By replacing $f$ with a translation as in Theorem~\ref{thm: main with all subtori f numerically periodic}, we may assume $f$ is a self-isogeny.

For (1), $A$ does not contain a subtorus of dimension one. For (2), every subtorus of $A$ is numerically $f^2$-stable, hence $f^2$-stable (cf.~Lemma \ref{lem: numerically parallel abelian varieties}). We also note that $f^2$ is $q$-polarized for some $q>2$. By Proposition~\ref{prop: uniform surface incompressible index stable}, we are done.
\end{proof}

\begin{remark}\label{rem:2-pol}
    If $f$ is a $2$-polarized self-isogeny of a simple abelian surface $A$ then as proved before, $\Ker(f)=\Z/2\Z\times \Z/2\Z$, and the eigenvalues of $f$ (acting on the universal cover of $A$) satisfy a monic polynomial with integer coefficients of the form $x^2-\alpha x+\beta$, where $|\beta|=2$ (see the proof of Proposition~\ref{prop: polarized rank simple surface} and Lemmas~\ref{lem: complex structure of endomorphism ring of abelian variety} and \ref{lem: real characteristic polynomial}). Since $f$ is polarized, the eigenvalues of the monic have the same modulus. It follows that either $\alpha=0$ and $\beta=-2$, or $\alpha^2-4\beta\leq 0$. In the latter case, $\beta=2$ and $|\alpha|\in \{0,1,2\}$. Thus $(\alpha,\beta)\in \{(0,-2), (0,2), (\pm 1,2), (\pm 2,2)\}$.

    By \cite[Proposition~7]{Led}, the essential dimension of $\Ker(f)$ over a characteristic zero field is two. This does not immediately determine $\mathrm{ed}(f)$.
\end{remark}

\section{Further questions}

\subsection{Essential dimensions of simple abelian varieties}

Recall that in Theorem~\ref{thm: simple abelian surface uniform index} and Proposition~\ref{prop: uniform surface incompressible index stable}, we proved that a polarized endomorphism on a simple abelian surface is incompressible except when $f$ is $2$-polarized. It is natural to ask the following question.

\begin{question}\label{que: strong restricted rank lower bound}
    Let $A$ be a simple abelian variety and let $f$ be a polarized endomorphism of $A$. Is it true that $f$ is incompressible?
\end{question}

In view of Theorem~\ref{thm: rank main uniform} and Remark~\ref{rem: kernel and incompressible}(1), we ask the following question. Note that an affirmative answer to 
Question~\ref{que: strong restricted rank lower bound} would imply an affirmative answer to this question.

\begin{question}
    Let $A$ be a simple abelian variety and let $f$ be a polarized endomorphism of $A$. Is it true that $\rank_p \Ker (f)\geq \dim(A)$ for any prime $p\mid \deg(f)$?
\end{question}

\medskip

\subsection{Essential dimensions of polarized endomorphisms in positive characteristic}

Zhuang pointed out the following example, originally due to Fakhruddin 
\cite[after Question~19]{KZ}, of a polarized endomorphism on $(\mathbf{P}^1)^n$ 
over a field of characteristic $p$ whose iterations all have essential dimension one.

\begin{example}\label{ex: positive characteristic iteration}
Consider the endomorphism $f$ of $(\mathbf{P}^1)^n$ given by 
$$(x_1,\ldots,x_n) \mapsto (x_1^p - x_1,\ldots,x_n^p - x_n),$$ so that $f$ restricted to each factor is an 
Artin--Schreier cover with Galois group $\Z/p\Z$. Since the essential dimension of 
an elementary abelian $p$-group over a field of characteristic $p$ is one 
(see \cite[Proposition 5]{Led}), we have $\mathrm{ed}(f)\leq \mathrm{ed}((\Z/p\Z)^n)=1$ by \cite[Theorem~3.1(c)]{BR97}. Moreover, for any iterate $f^s$ with $s>1$, the Galois group is $(\Z/p\Z)^{ns}$, which also has essential dimension one. Hence 
$\mathrm{ed}(f^s) = 1$ for all $s \geq 1$. 
\end{example}

\begin{question}
    Does Theorem~\ref{thm: ed main} hold for endomorphisms in characteristic $p$ when $\deg(f)$ is coprime to $p$?
\end{question}

\medskip

\subsection{Dynamical Manin--Mumford conjecture and essential dimension}\label{subsec: dynamic Manin Mumford}

Recall that the original version of the algebraic dynamical Manin--Mumford Conjecture \cite[Conjecture 2.1]{GTZ} asserts: \emph{Let $\phi: X\to X$ be a polarized endomorphism of a projective variety defined over the complex numbers, and
let $Y$ be a subvariety of $X$. Then $Y$ is preperiodic if and only if $Y \cap \mathrm{Prep}_\phi(X)$ is Zariski dense in $Y$.} Here, $\mathrm{Prep}_\phi(X)$ denotes the set of all $f$-preperiodic points in $X$. In the same paper, the authors provided counterexamples to the above-mentioned conjecture. Motivated by the above version, we introduce the following definition.

\begin{definition}[Dynamical Manin--Mumford condition]\label{defn: dynamical Manin Mumford condition}
    Let $f$ be a polarized endomorphism of a projective variety $X$. We say $f$ satisfies the {\it dynamical Manin--Mumford condition} if the following holds: for any subvariety $Y$ of $X$, the $Y$ is $f$-preperiodic if and only if $Y \cap \mathrm{Prep}_f(X)$ is Zariski dense in $Y$.
\end{definition}

The following gives an equivalent condition of the dynamical Manin--Mumford condition on abelian varieties.

\begin{lemma}\label{lem: dynamical Manin Mumford condition on abelian varieties}
    Let $A$ be an abelian variety, and let $f$ be a polarized isogeny of $A$. Then the following are equivalent.
    \begin{enumerate}
        \item $f$ satisfies the dynamical Manin--Mumford condition.
        \item Every subtorus of $A$ is $f$-preperiodic.
    \end{enumerate}
\end{lemma}

\begin{proof}
    By \cite[Claim~3.2]{GTZ}, $A_\mathrm{tor}=\mathrm{Prep}_f(A)$, where $A_\mathrm{tor}$ denotes the torsion points of $A$. In particular, $\mathrm{Prep}_f(A)\cap B=A_\mathrm{tor}\cap B=B_\mathrm{tor}$ is Zariski dense in $B$ for every subtorus $B\subset A$. Assuming (1), (2) follows directly from the definition of the dynamical Manin--Mumford condition. Now we assume (2) and prove (1). Let $B$ be a subvariety of $A$.

    If $B$ is $f$-preperiodic, then $f^r(B)=f^s(B)$ for some positive integers $r>s$. In particular, $f^{r-s}$ is a polarized endomorphism on $f^s(B)$. Hence
    $f^s(B) \cap \Prep_f(A)$ is Zariski dense in $f^s(B)$ by \cite[Theorem~5.1]{Fak03}. Then $B \cap (f^{-s}(f^s(B) \cap \Prep_f(A)))$ is Zariski dense in $B$, so the set $B\cap\mathrm{Prep}_f(A)$ is Zariski dense in $B$.

    

    If $\mathrm{Prep}_f(A)\cap B$ is Zariski dense in $B$, then $B$ contains a Zariski dense set of torsion points of $A$. Hence, by \cite{Ray83} $B$ is a torsion translate of a subtorus of $A$, that is, $B=\gamma+B_0$, where $\gamma\in A_\mathrm{tor}$ and $B_0$ is a subtorus of $A$. By Lemma~\ref{lem: uniform stable index}(1) and assumption (2), $B_0$ is $f^s$-stable for some positive integer $s$. Since $\gamma\in A_\mathrm{tor}$, as in \cite[Claim~3.2]{GTZ}, there are integers $r > t$ such that $f^{r}(\gamma)=f^t(\gamma)$. Thus, $f^{rs}(B)=f^{ts}(B)$, and hence $B$ is $f$-preperiodic.
\end{proof}

As a direct corollary of Theorem~\ref{thm: ed main} (Theorem~\ref{thm: main with all subtori f numerically periodic}) and Lemma~\ref{lem: dynamical Manin Mumford condition on abelian varieties}, we have the following.

\begin{theorem}\label{thm: dynamical Manin Mumford condition}
    Let $A$ be an abelian variety, and let $f$ be a polarized isogeny of $A$. Suppose that $f$ satisfies the dynamical Manin--Mumford condition. Then $f^s$ is incompressible for some positive integer $s$.
\end{theorem}

In view of Theorem~\ref{thm: dynamical Manin Mumford condition}, we propose the following question.

\begin{question}\label{ques: dynamical Manin Mumford condition}
    Let $X$ be a projective variety, and let $f$ be a polarized endomorphism of $X$. Suppose that $f$ satisfies the dynamical Manin--Mumford condition. Is $f^s$ incompressible for some positive integer $s$?
\end{question}

\end{document}